\documentclass{article}

\usepackage{amssymb}
\usepackage{latexsym}
\usepackage{amsmath}
\usepackage{mathrsfs}
\usepackage{color}
\usepackage{CJK}

\newtheorem{theorem}{Theorem}[section]
\newtheorem{corollary}{Corollary}[section]
\newtheorem{lemma}{Lemma}[section]
\newtheorem{proposition}{Proposition}[section]
\newtheorem{remark}{Remark}[section]
\newtheorem{definition}{Definition}[section]

\numberwithin{equation}{section}
\newenvironment{proof}{\medskip\par\noindent{\bf Proof.}\ }{\qquad
\raisebox{-0.5mm}{\rule{1.5mm}{4mm}}\vspace{6pt}}

\newcommand{\bbr}{\mathbb{R}}
\newcommand{\h}{H^1(\bbr^N)}

\newcommand{\ve}{\varepsilon}
\newcommand{\bd}{\begin{definition}}
\newcommand{\ed}{\end{definition}}

\newcommand{\br}{\begin{remark}}
\newcommand{\er}{\end{remark}}

\newcommand{\be}{\begin{equation}}
\newcommand{\ee}{\end{equation}}

\newcommand{\bc}{\begin{corollary}}
\newcommand{\ec}{\end{corollary}}
\begin{document}

\title
{\Large\bf Positive solutions to an elliptic equation in $\bbr^N$ of the Kirchhoff type}

\author{
Yisheng Huang$^{a},$\thanks{E-mail address: yishengh@suda.edu.cn(Yisheng Huang)}\quad
Zeng Liu$^{b},$\thanks{E-mail address: luckliuz@163.com(Zeng Liu)}\quad
Yuanze Wu$^{c}$\thanks{Corresponding
author. E-mail address: wuyz850306@cumt.edu.cn (Yuanze Wu).}\\%
\footnotesize$^{a}${\it  Department of Mathematics, Soochow University,}\\
\footnotesize{Suzhou 215006, P.R. China }\\%
\footnotesize$^{b}${\it  Department of Mathematics, Suzhou University of Science and Technology,}\\
\footnotesize{Suzhou 215009, P.R. China}\\
\footnotesize$^{c}${\it  College of Sciences, China University of Mining and Technology,}\\
\footnotesize{Xuzhou 221116, P.R. China }}%
\date{}
\maketitle

\begin{center}
\begin{minipage}{120mm}

\noindent{\bf Abstract:}   In this paper, we consider the following Kirchhoff type problem
$$
\left\{\aligned&-\bigg(a+b\int_{\bbr^N}|\nabla u|^2dx\bigg)\Delta u+V(x) u=|u|^{p-2}u&\text{ in }\bbr^N,\\%
&u\in\h,\endaligned\right.\eqno{(\mathcal{P}_{a,b})}%
$$
where $N\geq3$, $2<p<2^*=\frac{2N}{N-2}$, $a,b>0$ are parameters and $V(x)$ is a potential function.  Under some mild conditions on $V(x)$, we prove that $(\mathcal{P}_{a,b})$ has a positive solution for $b$ small enough by the variational method, a non-existence result is also established in the cases $N\geq4$.  Our results in the case $N=3$ partial improve the results in \cite{G15,LY14} and our results in the cases $N\geq4$ are totally new to the best of our knowledge.  By combining the scaling technique, we also give a global description on the structure of the positive solutions to the autonomous form of $(\mathcal{P}_{a,b})$, that is $V(x)\equiv\lambda>0$.  This result can be seen as a partial complement of the studies in \cite{A12,A13}.

\vspace{6mm} \noindent{\bf Keywords:} Positive solution; Kirchhoff problem; Variational method; Scaling technique.%

\vspace{6mm}\noindent {\bf AMS} Subject Classification 2010: 35B09; 35B38; 35J20; 35J61.%

\end{minipage}
\end{center}

\vskip0.26in

\section{Introduction}
In this paper, we consider the following Kirchhoff type problem
$$
\left\{\aligned&-\bigg(a+b\int_{\bbr^N}|\nabla u|^2dx\bigg)\Delta u+V(x)u=|u|^{p-2}u&\text{ in }\bbr^N,\\%
&u\in\h,\endaligned\right.\eqno{(\mathcal{P}_{a,b})}%
$$
where $N\geq3$, $2<p<2^*=\frac{2N}{N-2}$, $a,b>0$ are parameters and $V(x)$ is a potential function satisfying the following condition:
\begin{enumerate}
\item[$(V)$]  There exist two positive constants $v_0$ and $v_\infty$ such that
\begin{eqnarray*}
v_0\leq V(x)\leq v_\infty=\lim_{|y|\to+\infty}V(y)\quad \text{for all }x\in\bbr^N
\end{eqnarray*}
and there exists a subset of $\bbr^N$ with positive Lebesgue measure such that $V(x)<v_\infty$ on this set.
\end{enumerate}

The operator $-\bigg(a+b\int_{\Omega}|\nabla u|^2dx\bigg)\Delta u$ first appears in the following model:
\begin{equation}\label{eq001}
\left\{\aligned &u_{tt}-\bigg(a+b\int_{\Omega}|\nabla u|^2dx\bigg)\Delta u=h(x,u)\quad\text{in }\Omega\times(0, T),\\
&u=0\quad\text{on }\partial\Omega\times(0, T),\\
&u(x,0)=u_0(x),\quad u_t(x,0)=u^*(x),\endaligned\right.
\end{equation}
where $\Omega\subset\bbr^N$ is a bounded domain, $T>0$ is a constant, $u_0, u^*$ are continuous functions.  Such model was first proposed by Kirchhoff in 1883 as an extension of the classical D'Alembert's wave equations for free vibration of elastic strings, Kirchhoff's model takes into account the changes in length of the string produced by transverse vibrations.  Due to this reason, the operators as $-\bigg(a+b\int_{\Omega}|\nabla u|^2dx\bigg)\Delta u$ are always called as the Kirchhoff type operators, and the equations including the Kirchhoff type operators are always called as the Kirchhoff type problems.  In \eqref{eq001}, $u$ denotes the displacement, the nonlinearity $h(x,u)$ denotes the external force and the parameter $a$ denotes the initial tension while the parameter $b$ is related to the intrinsic properties of the string (such as Young¡¯s modulus).  For more details on the physical background of the Kirchhoff type problems, we refer the readers to \cite{A12,K83}.

Under some suitable assumptions on the nonlinearities, the elliptic type Kirchhoff problems have variational structures in some proper Hilbert spaces.  Due to this reason, such problems have been studied extensively in the literatures by the variational method, see for example \cite{A12,A13,ACM05,AF12,CWL12,CKW11,G15,HZ12,HLP14,HLW151,LLS12,LLS14,LY14,LLT152,N141,PZ06,R15,SW14,WHL151,W15} and the references therein.  In particular, in \cite{LY14}, Li and Ye studied the Kirchhoff type problem $(\mathcal{P}_{a,b})$ in $\bbr^3$.
Under some further assumptions on $V(x)$, the authors proved that $(\mathcal{P}_{a,b})$ has a positive ground state solution in $\bbr^3$ for $3<p<6$ by the variational method.  Their proof is dependent heavily on the following Pohozaev type condition:
\begin{enumerate}
\item[$(V_0)$] $V(x)\in C(\bbr^3, \bbr)$ is weakly differentiable and $(D V(x),x)\in L^\infty(\bbr^3)\cup L^{\frac32}(\bbr^3)$ and
\begin{eqnarray*}
V(x)-(D V(x), x)\geq0\quad\text{a.e. in }\bbr^3,
\end{eqnarray*}
where $(\cdot,\cdot)$ is the usual inner product in $\bbr^3$.
\end{enumerate}
Li and Ye's result was partially improved by Guo in \cite{G15}, where the following non-autonomous Kirchhoff type problem was considered
\begin{eqnarray}\label{eqnew0003}
\left\{\aligned&-\bigg(a+b\int_{\bbr^3}|\nabla u|^2dx\bigg)\Delta u+V(x)=f(u)&\text{ in }\bbr^3,\\
&u\in H^1(\bbr^3),\endaligned\right.
\end{eqnarray}
where $a,b>0$ are parameters, $V(x)$ is a potential function and $f(u)$ is a nonlinearity involving the power-type $|u|^{p-2}u$ for $2<p<6$.  Under some further assumptions on $V(x)$ and $f(u)$, Guo proved that \eqref{eqnew0003} has a positive ground state solution, in particular, when $f(u)=|u|^{p-2}u$, then \eqref{eqnew0003} has a positive ground state solution for all $2<p<6$.  Guo's proof is also dependent heavily on the following Pohozaev type condition:
\begin{enumerate}
\item[$(V_1)$] $V(x)\in C^1(\bbr^3, \bbr)$ and there exists a positive constant $A<a$ such that
\begin{eqnarray*}
|(\nabla V(x), x)|\leq\frac{A}{|x|^2}\quad\text{for all $x$ in }\bbr^3\backslash\{0\}.
\end{eqnarray*}
\end{enumerate}
Since \eqref{eqnew0003} involves $(\mathcal{P}_{a,b})$, a natural question inspired by the above facts is that
\begin{enumerate}
\item[$(Q_1)$] Are the Pohozaev type conditions as $(V_0)$ or $(V_1)$ necessary in finding the positive solution of $(\mathcal{P}_{a,b})$?
\end{enumerate}
It is worth to point out that, as pointed out by Li and Ye in \cite{LY14}, by using a similar method in \cite{HZ12}, we can obtain the following.
\begin{theorem}\label{thmnew001}
Let $N=3$ and $4<p<6$.  If $V(x)$ satisfies the condition $(V)$, then $(\mathcal{P}_{a,b})$ has a positive solution.
\end{theorem}
It follows from Theorem~\ref{thmnew001} that the question $(Q_1)$ has a positive answer for $N=3$ and $4<p<6$.  However, to the best of our knowledge, the question $(Q_1)$ for the cases $N\geq3$ and $p\in(2,4]\cap(2,2^*)$ is still open.  Thus, the first purpose of this paper is to explore the question $(Q_1)$ in the cases $N\geq3$ and $p\in(2,4]\cap(2,2^*)$.

Before we state our first result, we need to introduce some notations.  Let
\begin{eqnarray}\label{eq0104}
\mathcal{J}_V(u)=\frac{b}{4}\|\nabla u\|_{L^2(\bbr^N)}^4+\frac{a}{2}\|\nabla u\|_{L^2(\bbr^N)}^2+\frac12\int_{\bbr^N}V(x)u^2dx-\frac1p\|u\|_{L^p(\bbr^N)}^p,
\end{eqnarray}
where $\|\cdot\|_{L^q(\bbr^N)}$ is the usual norm in $L^q(\bbr^N)(q\geq1)$.  Then it is easy to check that $\mathcal{J}_V(u)$ is of $C^2$ in $\h$ and the critical points of $\mathcal{E}(u)$ are equivalent to the weak solutions to $(\mathcal{P}_{a,b})$ under the condition $(V)$.  Let the Nehari type manifold of $\mathcal{J}_V(u)$ be
\begin{eqnarray}\label{eq0102}
\mathcal{N}_V=\{u\in\h\backslash\{0\}\mid \mathcal{J}_V'(u)u=0\}.
\end{eqnarray}
Then it is easy to see that all nontrivial critical points are contained in $\mathcal{N}_V$.  Let
\begin{eqnarray*}
G_u(t)&=&\mathcal{J}_V(tu)\\
&=&\frac{bt^4}{4}\|\nabla u\|_{L^2(\bbr^N)}^4+\frac{at^2}{2}\|\nabla u\|_{L^2(\bbr^N)}^2\\
&&+\frac{t^2}2\int_{\bbr^N}V(x)u^2dx-\frac{t^p}{p}\|u\|_{L^p(\bbr^N)}^p.
\end{eqnarray*}
Then by a direct calculation, we can see that $G_u(t)$ is of $C^2$ in $\bbr^+$ for every $u\in\h$ and $G_u'(t)=0$ if and only if $tu\in\mathcal{N}_V$.  Thus, it is natural to divide the Nehari type manifold $\mathcal{N}_V$ into the following three parts:
\begin{eqnarray}
\mathcal{N}_V^-&=&\{u\in\mathcal{N}_V\mid G_u''(1)<0\};\label{eq1005}\\
\mathcal{N}_V^0&=&\{u\in\mathcal{N}_V\mid G_u''(1)=0\};\\
\mathcal{N}_V^+&=&\{u\in\mathcal{N}_V\mid G_u''(1)>0\}.\label{eq1007}
\end{eqnarray}
Now, our first result can be stated as follows.
\begin{theorem}\label{thm1001}
Let $N\geq3$, $a>0$, $p\in(2, 4]\cap(2, 2^*)$ and $V(x)$ satisfy the condition $(V)$.  Then there exist $0<b_*(a)\leq b_{**}(a)<+\infty$ such that $(\mathcal{P}_{a,b})$ has a positive solution for $0<b<b_*(a)$, which minimizes the functional $\mathcal{J}_V(u)$ on $\mathcal{N}_V^-$.  Moreover, $(\mathcal{P}_{a,b})$ only has trivial solutions for $b>b_{**}(a)$ in the cases $N\geq4$.
\end{theorem}
\begin{remark}
\begin{enumerate}
\item[$(1)$]  For the case $N=3$, Theorem~\ref{thm1001} can be seen as an improvement of the results in \cite{G15,LY14} to $(\mathcal{P}_{a,b})$ in the sense that we totally remove the conditions $(V_0)$ or $(V_1)$ for $b$ small enough to obtain positive solutions of $(\mathcal{P}_{a,b})$.  To the best of our knowledge, Theorem~\ref{thm1001} is totally new for the cases $N\geq4$.
\item[$(2)$]  By Theorem~\ref{thm1001}, we can see that the Pohozaev type conditions are not needed in finding positive solutions of $(\mathcal{P}_{a,b})$ for the parameter $b$ small enough.  Thus, Theorem~\ref{thm1001} gives a partial answer to the question $(Q_1)$.
\item[$(3)$]  It is also worth to point that in \cite{LLS12}, Li et al. studied the existence of a positive solution for the following autonomous Kirchhoff problem
\begin{eqnarray*}
\left\{\aligned&\bigg(a+b(\int_{\bbr^N}(|\nabla u|^2+\lambda u^2)dx)\bigg)(-\Delta u+\lambda u)=f(u)&\text{ in }\bbr^N,\\%
&u\in\h,\endaligned\right.
\end{eqnarray*}
where $N\geq3$, $a,b>0$ are parameters, $\lambda>0$ is a constant and $f(u)$ is a nonlinearity involving the power-type $|u|^{p-2}u$ for $2<p<2^*$.  By using a truncation argument combined with a monotonicity trick introduced by Jeanjean \cite{J99} (see also Struwe \cite{S85}), the authors proved that such equation has a positive radial solution for $b$ small enough.  Their method is heavily dependent on two facts.  One fact is that $H^1_r(\bbr^N)$ is embedded compactly into $L^q(\bbr^N)$ for $2\leq q<2^*$, where $H^1_r(\bbr^N)=\{u\in\h\mid u\text{ is radial symmetric}\}$.  The other fact is that the order $4$ term $(\|\nabla u\|_{L^2(\bbr^N)}^2+\|u\|_{L^2(\bbr^N)}^2)^2$ can totally control the nonlinearity $f(u)$ by the Sobolev embedding theorem.  Since $D^{1,2}(\bbr^N)$ can not be embedded into $L^p(\bbr^N)$, the order $4$ term $\|\nabla u\|_{L^2(\bbr^N)}^4$ can not control the order $p$ term $\|u\|_{L^p(\bbr^N)}^p$ totally for $(\mathcal{P}_{a,b})$.  On the other hand, since $(\mathcal{P}_{a,b})$ is non-autonomous, $H^1_r(\bbr^N)$ is not a good choice for the variational setting of $(\mathcal{P}_{a,b})$.  Due to these reasons, their method can not be used for $(\mathcal{P}_{a,b})$.
\item[$(4)$]  From the view point of the fibering maps, it seems that there exists another positive solution to $(\mathcal{P}_{a,b})$ for $b$ small enough in the cases $2<p<\min\{4, 2^*\}$, which minimizes the functional $\mathcal{J}_V(u)$ on $\mathcal{N}_V^+$.  However, we actually observe in Theorem~\ref{thm0004} below that for the autonomous form of $(\mathcal{P}_{a,b})$, there exists a unique positive solution in the cases $N=3,4$.  Thus, it seems that $\inf_{\mathcal{N}_V^+}\mathcal{J}_V(u)$ can not be attained in the cases $N=3,4$.  For the cases $N\geq5$, we believe that there exists another positive solution to $(\mathcal{P}_{a,b})$ for $b$ small enough, which minimizes the functional $\mathcal{J}_V(u)$ on $\mathcal{N}_V^+$.  However, since the energy values of bubbles to the $(PS)$ sequence of $\mathcal{J}_V(u)$ may be negative at the energy level $\inf_{\mathcal{N}_V^+}\mathcal{J}_V(u)$, it is hard to exclude the dichotomy case in the concentration-compactness principle (see Lions \cite{L84}).  Due to this reason, we do not obtain the second solution to $(\mathcal{P}_{a,b})$ for $b$ small enough in the cases $N\geq5$.
\item[$(5)$]  The condition $(V)$ can be weaken to some other ones which ensure that the Schr\"odinger operator $-\Delta+V(x)$ is definite.  However, we do not want to go further in that direction.
\item[$(6)$]  Even though the positive solution obtained by Theorem~\ref{thm1001} minimizes the functional $\mathcal{J}_V(u)$ on $\mathcal{N}_V^-$, this solution also may not be a ground state solution for $(\mathcal{P}_{a,b})$, since $\mathcal{N}_V^0$ and $\mathcal{N}_V^+$ may be nonempty sets.  Thus, how to find the ground state solution of $(\mathcal{P}_{a,b})$ without the Pohozaev type conditions as $(V_0)$ or $(V_1)$ is still an interesting question and also open to us.
\end{enumerate}
\end{remark}

On the other hand, Azzollini introduced the scaling technique
\begin{eqnarray*}
u(x)\to u_t(x):=u(tx)
\end{eqnarray*}
to deal with the autonomous Kirchhoff type problem in $\bbr^N$ in \cite{A12,A13}.  Such method is much more simple than the variational method and can be used to establish the relation between the solutions of the autonomous Kirchhoff type problem and that of the related local problem.  Applying Azzollini's scaling technique to the autonomous form of $(\mathcal{P}_{a,b})$, we can easily to obtain the following.
\begin{theorem}\label{thm1002}
Let $N\geq3$, $a,b>0$, $2<p<2^*$ and $V(x)\equiv\lambda>0$.  Then we have the following.
\begin{enumerate}
\item[$(i)$] In the case $N=3$, the autonomous form of $(\mathcal{P}_{a,b})$ has a unique positive radial solution $u_{a,b,\lambda}$ with the expression
\begin{eqnarray}\label{eq2008}
u_{a,b,\lambda}(x)=U_\lambda(tx),
\end{eqnarray}
where $t$ is a positive constant satisfying
\begin{eqnarray*}
at^2+b\bigg(\int_{\bbr^N}|U_{\lambda}|^2dx\bigg)t^{4-N}=1
\end{eqnarray*}
and $U_\lambda$ is the unique positive radial solution of the equation
\begin{eqnarray}\label{eqnew1006}
\left\{\aligned&-\Delta u+\lambda u=|u|^{p-2}u&\text{ in }\bbr^3,\\
&u\in H^1(\bbr^3).\endaligned\right.
\end{eqnarray}
\item[$(ii)$] In the case $N=4$, if $b\int_{\bbr^4}|\nabla U_\lambda|^2dx<1$, then the autonomous form of $(\mathcal{P}_{a,b})$ has a unique positive radial solution $u_{a,b,\lambda}$ with the same expression of \eqref{eq2008}, where $U_\lambda$ is the unique positive radial solution of the equation \eqref{eqnew1006} in $\bbr^4$.  If $b\int_{\bbr^4}|\nabla U_\lambda|^2dx\geq1$, then
             the autonomous form of $(\mathcal{P}_{a,b,\lambda})$ has no solution.
\item[$(iii)$] In the cases $N\geq5$, if $\mathcal{F}_{a,b}(U_\lambda)<1$,
             then the autonomous form of $(\mathcal{P}_{a,b})$ has exact two positive radial solutions $u_{a,b,\lambda}^\pm$ with the same expressions of \eqref{eq2008}, where
             \begin{eqnarray*}
             \mathcal{F}_{a,b}(U_\lambda)=\bigg(\frac{(N-4)b\int_{\bbr^N}|\nabla U_{\lambda}|^2dx}{2}\bigg)^{\frac{2}{N-2}}\frac{(N-2)a^{\frac{N-4}{N-2}}}{N-4}
             \end{eqnarray*}
             and $U_\lambda$ is the unique positive radial solution of the equation \eqref{eqnew1006} in $\bbr^N(N\geq5)$.  If $\mathcal{F}_{a,b}(U_\lambda)=1$,
             then the autonomous form of $(\mathcal{P}_{a,b})$ has a unique positive radial solution $u_{a,b,\lambda}^0$ with same expression of \eqref{eq2008}, where
             $U_\lambda$ is the unique positive radial solution of the equation \eqref{eqnew1006} in $\bbr^N(N\geq5)$.  If $\mathcal{F}_{a,b}(U_\lambda)>1$,
             then the autonomous form of $(\mathcal{P}_{a,b})$ has no solution.
\end{enumerate}
\end{theorem}
Since the autonomous form of $(\mathcal{P}_{a,b})$ can also be studied by the variational method, a natural question for the autonomous form of $(\mathcal{P}_{a,b})$ due to Theorem~\ref{thm1002} is that
\begin{enumerate}
\item[$(Q_2)$] Can the solutions to the autonomous form of $(\mathcal{P}_{a,b})$ founded in Theorem~\ref{thm1002} also be founded by the variational method, that is, does the solutions to the autonomous form of $(\mathcal{P}_{a,b})$ founded by the scaling technique coincide with that founded by the variational method?
\end{enumerate}
Due to the uniqueness of the positive solution to $(\mathcal{P}_{a,b})$ given by Theorem~\ref{thm1002}, this question has a positive answer for the cases $N=3,4$ by the results in \cite{A12}.  However,
since the autonomous form of $(\mathcal{P}_{a,b})$ has two positive solutions for $\mathcal{F}_{a,b}(U_\lambda)<1$ due to Theorem~\ref{thm1002}, this question is still open for the cases $N\geq5$ to the best of our knowledge.  Thus, the second purpose of this paper is to study the question $(Q_2)$ to the autonomous form of $(\mathcal{P}_{a,b})$ for the cases $N\geq5$.

Before we state our results on the question $(Q_2)$, we also need to introduce some notations.  Let
\begin{eqnarray}\label{eq0004}
\mathcal{E}(u)=\frac{b}{4}\|\nabla u\|_{L^2(\bbr^N)}^4+\frac{a}{2}\|\nabla u\|_{L^2(\bbr^N)}^2+\frac{\lambda}{2}\|u\|_{L^2(\bbr^N)}^2-\frac1p\|u\|_{L^p(\bbr^N)}^p.
\end{eqnarray}
Then it is easy to check that $\mathcal{E}(u)$ is of $C^2$ in $\h$ and the critical points of $\mathcal{E}(u)$ are equivalent to the weak solutions to the autonomous form of $(\mathcal{P}_{a,b})$.  Let the Pohozaev type manifold of $\mathcal{E}(u)$ be
\begin{eqnarray}\label{eq0002}
\mathcal{M}=\{u\in\h\backslash\{0\}\mid \Psi(u)=0\},
\end{eqnarray}
where
\begin{eqnarray}
\Psi(u)&=&\frac{N-2}{2N}(a\|\nabla u\|_{L^2(\bbr^N)}^2+b\|\nabla u\|_{L^2(\bbr^N)}^4)\notag\\
&&+\frac{\lambda}{2}\|u\|_{L^2(\bbr^N)}^2-\frac{1}{p}\|u\|_{L^p(\bbr^N)}^p.\label{eq0003}
\end{eqnarray}
Then by the Pohozaev identity of the autonomous form of $(\mathcal{P}_{a,b})$ (see \cite{A12,LY14}), every critical point of $\mathcal{E}(u)$ is contained in $\mathcal{M}$.  Let
\begin{eqnarray*}
F_u(t)&=&\mathcal{E}(u_t)\\
&=&\frac{b}{4}\|\nabla u\|_{L^2(\bbr^N)}^4t^{4-2N}+\frac{a}{2}\|\nabla u\|_{L^2(\bbr^N)}^2t^{2-N}\\
&&+(\frac{\lambda}{2}\|u\|_{L^2(\bbr^N)}^2-\frac{1}{p}\|u\|_{L^p(\bbr^N)}^p)t^{-N},
\end{eqnarray*}
where $u_t(x)=u(tx)$.  Then by a direct calculation, we can see that $F_u(t)$ is of $C^2$ in $\bbr^+$ for every $u\in\h$ and $F_u'(t)=0$ if and only if $u_t\in\mathcal{M}$.  Thus, it is natural to divide the Pohozaev type manifold $\mathcal{M}$ given by \eqref{eq0002} into the following three parts:
\begin{eqnarray}
\mathcal{M}^-&=&\{u\in\mathcal{M}\mid F_u''(1)<0\};\label{eq0005}\\
\mathcal{M}^0&=&\{u\in\mathcal{M}\mid F_u''(1)=0\};\label{eq0006}\\
\mathcal{M}^+&=&\{u\in\mathcal{M}\mid F_u''(1)>0\}.\label{eq0007}
\end{eqnarray}
Now, our second result can be stated as follows.
\begin{theorem}\label{thm0001}
Let $N\geq5$, $a,b>0$, $2<p<2^*$ and $V(x)\equiv\lambda>0$.  Let $U_\lambda$ be the unique positive radial solution of \eqref{eqnew1006} in $\bbr^N(N\geq5)$.  Then we have the following.
\begin{enumerate}
\item[$(i)$] If $\mathcal{F}_{a,b}(U_\lambda)<1$,
             then the autonomous form of $(\mathcal{P}_{a,b})$ has two positive radial solutions $u_{a,b,\lambda}^\pm$.  Moreover, we also have
             \begin{eqnarray*}
             \mathcal{E}(u_{a,b,\lambda}^-)=\inf_{\mathcal{M}^-}\mathcal{E}(u)\quad\text{and}\quad
             \mathcal{E}(u_{a,b,\lambda}^+)=\inf_{\mathcal{M}^+}\mathcal{E}(u)=\inf_{\mathcal{M}}\mathcal{E}(u).
             \end{eqnarray*}
\item[$(ii)$] If $\mathcal{F}_{a,b}(U_\lambda)=1$,
             then the autonomous form of $(\mathcal{P}_{a,b})$ has a positive radial solution $u_{a,b,\lambda}^0$.  Moreover, $\mathcal{M}=\mathcal{M}^0$ and $\mathcal{E}(u_{a,b,\lambda}^0)=\inf_{\mathcal{M}}\mathcal{E}(u)$.
\item[$(iii)$] If $\mathcal{F}_{a,b}(U_\lambda)>1$,
             then the autonomous form of $(\mathcal{P}_{a,b})$ has no solution.
\end{enumerate}
\end{theorem}
\begin{remark}
\begin{enumerate}
\item[$(1)$]  The proof of Theorem~\ref{thm0001} is pure variational.  Thus, Theorem~\ref{thm0001} gives a positive answer of the question $(Q_2)$ in the cases $N\geq5$.
\item[$(2)$]  In \cite{A12}, it has been proved that the ground state solution to the autonomous form of $(\mathcal{P}_{a,b})$ in the cases $N=3,4$ also minimizes the energy functional $\mathcal{E}(u)$ on the Pohozaev type manifold $\mathcal{M}$.  However, to the best of our knowledge, such property to the autonomous form of $(\mathcal{P}_{a,b})$ in the cases $N\geq5$ has not been obtained in the literatures.  Now, since every critical point of $\mathcal{E}(u)$ is contained in $\mathcal{M}$ by the Pohozaev identity of the autonomous form of $(\mathcal{P}_{a,b})$, $u_{a,b,\lambda}^+$ and $u_{a,b,\lambda}^0$ obtained by Theorem~\ref{thm0001} must be the the ground state solution to the autonomous form of $(\mathcal{P}_{a,b})$ respectively in the cases $\mathcal{F}_{a,b}(U_\lambda)<1$ and $\mathcal{F}_{a,b}(U_\lambda)=1$.  Thus, by Theorem~\ref{thm0001}, we can see that the ground state solution to the autonomous form of $(\mathcal{P}_{a,b})$ in the cases $N\geq5$ also minimizes the energy functional $\mathcal{E}(u)$ on the Pohozaev type manifold $\mathcal{M}$.
\item[$(3)$] Some other $C^1$-manifolds were used to find the positive ground state solution to the autonomous form of $(\mathcal{P}_{a,b})$ in \cite{G15,LY14} for the case $N=3$.  Such manifolds can be seen as some kinds of the unifications of the Nehari type manifold and the Pohozaev type manifold.  Thus, the energy level of the ground state solution to the autonomous form of $(\mathcal{P}_{a,b})$ in the case $N=3$ has some other expressions in the view point of the calculus of variation.  However, such manifolds are not good choices for the high dimensions $(N\geq4)$ due to the fact that $D^{1,2}(\bbr^N)$ can not be embedded into $L^p(\bbr^N)$.  Indeed, if we consider such manifolds, then we will trap in the trouble that we can not describe the manifold very clear for all $a,b>0$ in the high dimensions $(N\geq4)$.  Thus, we can not describe the positive solution to the autonomous form of $(\mathcal{P}_{a,b})$ totally as Theorem~\ref{thm0001}.
\end{enumerate}
\end{remark}

In \cite{CR92,CL97}, Chipot et al. introduced another scaling technique $u(x)\to tu(x)$ to deal with the elliptic equations of the Kirchhoff type with power-type nonlinearity (see also \cite{A13,ACM05,HLW15,LLT152}).  This method was further developed in our previous paper \cite{WHL15}.  Note that the nonlinearity to the autonomous form of $(\mathcal{P}_{a,b})$ is also power-type.  Thus, the autonomous form of $(\mathcal{P}_{a,b})$ also can be studied by the scaling technique $u(x)\to tu(x)$.  Due to this fact, the following question is also natural.
\begin{enumerate}
\item[$(Q_3)$]  What is the relation between the two differential scaling technique for the autonomous form of $(\mathcal{P}_{a,b})$?
\end{enumerate}
In order to study the question $(Q_3)$, we introduce a more general scaling technique $u\to su(tx)$, $s,t>0$, which can be seen as a unification of the two differential scaling technique used in the literatures.  By such scaling technique, we observe the following.
\begin{theorem}\label{thm0002}
Let $N\geq3$, $a,b>0$, $2<p<2^*$ and $V(x)\equiv\lambda>0$.  Then the solution of the autonomous form of $(\mathcal{P}_{a,b})$ must be of the form $sU_{\frac{\lambda}{s^{p-2}}}(tx)$, where $U_{\frac{\lambda}{s^{p-2}}}$ is the unique positive radial solution of \eqref{eqnew1006} for $\lambda=\frac{\lambda}{s^{p-2}}$ in $\bbr^N(N\geq3)$, $s$ and $t$ satisfy $\frac{t}{s^{\frac{p-2}{2}}}=\gamma>0$ and $\gamma$ is the solution of the following equation
\begin{eqnarray}\label{eqnew0005}
a\gamma^2+b\bigg(\int_{\bbr^N}|U_{\lambda}|^2dx\bigg)\gamma^{4-N}=1.
\end{eqnarray}
Moreover, the number of positive solutions to the autonomous form of $(\mathcal{P}_{a,b})$ equals to the number of solutions to the equation \eqref{eqnew0005}.
\end{theorem}

\begin{remark}
Theorem~\ref{thm0002} gives all expressions of the solutions to the autonomous form of $(\mathcal{P}_{a,b})$ obtained by the scaling technique.  Furthermore, note that it is well known that $U_{\lambda}(x)=\lambda^{\frac{1}{p-2}}U_1(\sqrt{\lambda}x)$.  Thus, we must have
\begin{eqnarray}\label{eqnew0010}
sU_{\frac{\lambda}{s^{p-2}}}(tx)=\lambda^{\frac{1}{p-2}}U_1((\frac{\lambda}{s^{p-2}})^{\frac12}tx)=U_\lambda(\frac{1}{s^{\frac{p-2}{2}}}tx)
=U_\lambda(\gamma x).
\end{eqnarray}
It follows that all scaling technique coincide with the special one $u(x)\to u(tx)$ for the autonomous form of $(\mathcal{P}_{a,b})$ due to Theorem~\ref{thm1002}, which also gives the answer to the question $(Q_3)$.
\end{remark}

Combing Theorems~\ref{thm0001}--\ref{thm0002} and the results in \cite{A12,A13}, we can obtain the following.
\begin{theorem}\label{thm0004}
Let $N\geq3$£¬ $a,b>0$, $2<p<2^*$ and $V(x)\equiv\lambda>0$.  Then we have the following.
\begin{enumerate}
\item[$(i)$] In the case $N=3$, the autonomous form of $(\mathcal{P}_{a,b})$ has a unique positive radial solution $u_{a,b,\lambda}$ with the expression
\begin{eqnarray}\label{eq0008}
u_{a,b,\lambda}(x)=\bigg(\frac{\lambda}{\alpha}\bigg)^{\frac{1}{p-2}}U_\alpha(tx),
\end{eqnarray}
where $\alpha$ and $t$ are two positive constants satisfying $\bigg(\frac{\alpha}{\lambda}\bigg)^{\frac12}t=\gamma>0$ and $\gamma$ is the solution of \eqref{eqnew0005}, moreover, $u_{a,b,\lambda}$ is the ground state solution and $\mathcal{E}(u_{a,b,\lambda})=\inf_{\mathcal{M}}\mathcal{E}(u)$.
\item[$(ii)$] In the case $N=4$, if $b\int_{\bbr^N}|\nabla U_\lambda|^2dx<1$, then the autonomous form of $(\mathcal{P}_{a,b})$ has a unique positive radial solution $u_{a,b,\lambda}$ with the same expression of \eqref{eq0008}, moreover, $u_{a,b,\lambda}$ is the ground state solution and $\mathcal{E}(u_{a,b,\lambda})=\inf_{\mathcal{M}}\mathcal{E}(u)$; if $b\int_{\bbr^N}|\nabla U_\lambda|^2dx\geq1$, then
             the autonomous form of $(\mathcal{P}_{a,b})$ has no solution.
\item[$(iii)$] In the cases $N\geq5$, if $\mathcal{F}_{a,b}(U_\lambda)<1$,
             then the autonomous form of $(\mathcal{P}_{a,b})$ has exact two positive radial solutions $u_{a,b,\lambda}^\pm$ with the same expressions of \eqref{eq0008}, moreover, $u_{a,b,\lambda}^+$ is the ground state solution and $\mathcal{E}(u_{a,b,\lambda}^-)=\inf_{\mathcal{M}^-}\mathcal{E}(u)$ and $\mathcal{E}(u_{a,b,\lambda}^+)=\inf_{\mathcal{M}^+}\mathcal{E}(u)=\inf_{\mathcal{M}}\mathcal{E}(u)$;
             if $\mathcal{F}_{a,b}(U_\lambda)=1$,
             then the autonomous form of $(\mathcal{P}_{a,b})$ has a unique positive radial solution $u_{a,b,\lambda}^0$ with same expression of \eqref{eq0008}, moreover, $u_{a,b,\lambda}^0$ is the ground state solution and $\mathcal{M}=\mathcal{M}^0=\{u_{a,b,\lambda}^0\}$; if $\mathcal{F}_{a,b}(U_\lambda)>1$,
             then the autonomous form of $(\mathcal{P}_{a,b})$ has no solution.
\end{enumerate}
\end{theorem}

\begin{remark}
\begin{enumerate}
\item[$(1)$]  To the best of our knowledge, Theorem~\ref{thm0004} is the first result which describe the positive solutions to the autonomous form of $(\mathcal{P}_{a,b})$ totally.  Thus, Theorems~\ref{thm0001}--\ref{thm0002} can be seen as a complement of the results in \cite{A12,A13} for the autonomous form of $(\mathcal{P}_{a,b})$.
\item[$(2)$]  By making some further observations on the function \eqref{eqnew0005}, we can obtain some concentration behaviors of the positive solutions to the autonomous form of $(\mathcal{P}_{a,b})$ for the parameters $a,b$ due to the precise expressions given by Theorem~\ref{thm0004}.  However, we do not want to go further in that direction in the current paper.
\end{enumerate}
\end{remark}

Through this paper, $o_n(1)$ will always denote the quantities tending to zero as $n\to\infty$ and $C_i$ will denote the positive constants which may be different and independent of the parameter $b$.  For the sake of simplicity, we respectively denote $\|\nabla u\|_{L^2(\bbr^N)}^2$, $\|u\|_{L^p(\bbr^N)}^p$, $\|u\|_{L^2(\bbr^N)}^2$ and $\int_{\bbr^N}V(x)u^2dx$ by $\mathfrak{A}_u$, $\mathfrak{B}_u$, $\mathfrak{C}_u$ and $\mathfrak{C}_{u,V}$ in the remaining of this paper.

\section{The autonomous case}
\subsection{The Pohozaev manifold $\mathcal{M}$ in $N\geq5$}
As we stated in the introduction, the fibering map $F_u(t)=\mathcal{E}(u_t)$ can be used to observe the the Pohozaev manifold $\mathcal{M}$ and the divisions $\mathcal{M}^\pm, \mathcal{M}^0$, where $u_t(x)=u(tx)$.  Let
\begin{eqnarray*}
\mathcal{C}&=&\{u\in\h\backslash\{0\}\mid\frac1p\mathfrak{B}_u-\frac\lambda2\mathfrak{C}_u>0\}
\end{eqnarray*}
and
\begin{eqnarray*}
\mathcal{M}^*&=&\{u\in\h\backslash\{0\}\mid\frac{N-2}{2N}\mathfrak{A}_u+\frac\lambda2\mathfrak{C}_u-\frac1p\mathfrak{B}_u=0\}.
\end{eqnarray*}
Then we have the following.
\begin{lemma}\label{lem0002}
For every $u\in\h\backslash\{0\}$, there exists a unique $t>0$ such that $u_t=u(tx)\in\mathcal{M}^*$ in the case $u\in\mathcal{C}$ while $u_t=u(tx)\not\in\mathcal{M}$ for all $t>0$ in the case $u\not\in\mathcal{C}$.  Moreover, if $u\in\mathcal{C}$, then $I(u_t)=\max_{s>0}I(u_s)$ and $I(u_s)$ is strictly increasing on $(0, t)$ and strictly decreasing on $(t, +\infty)$, , where $I(u)$ is the corresponding functional of \eqref{eqnew1006} and given by
$I(u)=\frac12\|\nabla u\|_{L^2(\bbr^N)}^2+\frac{\lambda}{2}\|u\|_{L^2(\bbr^N)}^2-\frac{1}{p}\|u\|_{L^p(\bbr^N)}^p$.
\end{lemma}
\begin{proof}
Let $u\in \h\backslash\{0\}$ and consider the fibering map $L_u(t)=I(u_t)$, where $u_t=u(tx)$.  By a direct calculation, we can see that
\begin{eqnarray}\label{eq0013}
L_u'(t)=Nt^{-N-1}(\frac1p\mathfrak{B}_u-\frac\lambda2\mathfrak{C}_u-\frac{N-2}{2N}\mathfrak{A}_ut^2).
\end{eqnarray}
Clearly, there exists a unique $t>0$ such that $L_u'(t)=0$ in the case $u\in\mathcal{C}$; while $L_u'(t)<0$ for all $t>0$ in the case $u\not\in\mathcal{C}$.  It follows from $t^{-N}\mathfrak{B}_u=\mathfrak{B}_{u_t}$, $t^{-N}\mathfrak{C}_u=\mathfrak{C}_{u_t}$ and $t^{2-N}\mathfrak{A}_u=\mathfrak{A}_{u_t}$ that $L_u'(t)=t^{-1}L_{u_t}'(1)$, which implies that there exists a unique $t>0$ such that $u_t=u(tx)\in\mathcal{M}^*$ in the case $u\in\mathcal{C}$; while $u_t=u(tx)\not\in\mathcal{M}$ for all $t>0$ in the case $u\not\in\mathcal{C}$.  Now, by \eqref{eq0013}, we can see that if $u\in\mathcal{C}$, then $I(u_t)=\max_{s>0}I(u_s)$ and $I(u_s)$ is strictly increasing on $(0, t)$ and strictly decreasing on $(t, +\infty)$.
\end{proof}

Let
\begin{eqnarray}
&\mathcal{B}_{-}=\{u\in\mathcal{C}\mid \mathcal{B}(u)<1\};\label{eq0022}\\
&\mathcal{B}_{0}=\{u\in\mathcal{C}\mid \mathcal{B}(u)=1\};\label{eqnew0011}\\
&\mathcal{B}_{+}=\{u\in\mathcal{C}\mid \mathcal{B}(u)>1\},\label{eqnew0012}
\end{eqnarray}
where
\begin{eqnarray}
\mathcal{B}(u)=\frac{a^{\frac{N-4}{N-2}}b^{\frac{2}{N-2}}\mathfrak{A}_u^{\frac{N}{N-2}}}
{2^{\frac{N}{N-2}}N(N-4)^{\frac{N-4}{N-2}}(\frac1p\mathfrak{B}_u-\frac\lambda2\mathfrak{C}_u)}.\label{eq0023}
\end{eqnarray}
Then our first observation on $\mathcal{M}$ can be stated as follows.
\begin{lemma}\label{lem0001}
Let $N\geq5$.  Then we have the following.
\begin{enumerate}
\item[$(1)$] For every $u\in\mathcal{B}_{-}$, there exist unique $0<t_+<t_-$ such that $u_{t,-}=u(t_-x)\in\mathcal{M}^-$ and $u_{t,+}=u(t_+x)\in\mathcal{M}^+$, where $\mathcal{M}^\pm$ are respectively given by \eqref{eq0005} and \eqref{eq0007}.
\item[$(2)$] For every $u\in\mathcal{B}_{0}$, there exists a unique $t>0$ such that $u_t=u(tx)\in\mathcal{M}^0$, where $\mathcal{M}^0$ is given by \eqref{eq0006}.
\item[$(3)$]  For every $u\in\mathcal{B}_{+}$, $u_t=u(tx)\not\in\mathcal{M}$ for all $t>0$.
\end{enumerate}
\end{lemma}
\begin{proof}
Let $u\in \h\backslash\{0\}$ and consider the fibering map $F_u(t)$.  By a direct calculation, we have that $F_u'(t)=Nt^{-N-1}F_{1,u}(t)$, where
\begin{eqnarray*}
F_{1,u}(t)=\frac1p\mathfrak{B}_u-\frac\lambda2\mathfrak{C}_u-\frac{N-2}{2N}\mathfrak{A}_u(at^2+b\mathfrak{A}_ut^{4-N}).
\end{eqnarray*}
It is easy to see that $F_u'(t)<0$ for all $t>0$ when $u\not\in\mathcal{C}$.  For every $u\in\mathcal{C}$, by Lemma~\ref{lem0002}, there exists $t_0>0$ such that $u_{t_0}\in\mathcal{M}^*$.  Thus,
\begin{eqnarray*}
F_{1,u}(t)=t_0^NF_{1,u_{t_0}}(s)=t_0^N\frac{N-2}{2N}\mathfrak{A}_{u_{t_0}}(1-b\mathfrak{A}_{u_{t_0}}s^{4-N}-as^2),
\end{eqnarray*}
where $s=\frac{t}{t_0}$.  Set $F_{2, u_{t_0}}(s)=1-b\mathfrak{A}_{u_{t_0}}s^{4-N}-as^2$.  Then by a direct calculation, we can see that there exist unique $0<s_1<s_2$ such that $F_{2, u_{t_0}}(s_i)=0$ for $i=1,2$, $F_{2, u_{t_0}}(s)<0$ for $0<s<s_1$, $F_{2, u_{t_0}}(s)>0$ for $s_1<s<s_2$ and $F_{2, u_{t_0}}(s)<0$ for $s>s_2$ when $\mathcal{F}_{a,b}(u_{t_0})<1$, there exists a unique $s_0=\bigg(\frac{(N-4)b\mathfrak{A}_{u_{t_0}}}{2a}\bigg)^{\frac{1}{N-2}}$ such that $F_{2, u_{t_0}}(s_0)=0$ when $\mathcal{F}_{a,b}(u_{t_0})=1$ and $F_{2, u_{t_0}}(s)<0$ for all $s>0$ when $\mathcal{F}_{a,b}(u_{t_0})>1$.  By \eqref{eq0013}, we can see that $t_0=\bigg(\frac{\frac1p\mathfrak{B}_u-\frac\lambda2\mathfrak{C}_u}{\frac{N-2}{2N}\mathfrak{A}_u}\bigg)^{\frac12}$.  Now, by a direct calculation, we have
\begin{eqnarray}
\mathcal{F}_{a,b}(u_{t_0})=\frac{a^{\frac{N-4}{N-2}}b^{\frac{2}{N-2}}\mathfrak{A}_u^{\frac{N}{N-2}}}
{2^{\frac{N}{N-2}}N(N-4)^{\frac{N-4}{N-2}}(\frac1p\mathfrak{B}_u-\frac\lambda2\mathfrak{C}_u)}=\mathcal{B}(u).\label{eq0053}
\end{eqnarray}
Thus, the conclusions follows immediately from the relation between the Pohozaev manifold $\mathcal{M}$ and the fibering map $F_u(t)$.
\end{proof}

\begin{remark}\label{rmk0001}
By checking the proof of Lemma~\ref{lem0001}, we can also see that for every $u\in\mathcal{B_{-}}$, $F_u(t_+)=\min_{0<s\leq t_-}F_u(s)$ and $F_u(t_-)=\max_{t_+\leq s}F_u(s)$ and $F_u(s)$ is strictly decreasing for $0<s<t_+$, strictly increasing for $t_+<s<t_-$ and strictly decreasing for $s>t_-$, where $t_\pm$ are given in Lemma~\ref{lem0001}.
\end{remark}

Our second observation on the Pohozaev manifold $\mathcal{M}$ is the following.
\begin{lemma}\label{lem0003}
Let $u\in\mathcal{M}$ and $N\geq5$, then we have
\begin{eqnarray}\label{eq0015}
\bigg(\frac{(N-2)ap}{2N}\bigg(\frac{\lambda p}{2}\bigg)^{\frac{2^*-p}{p-2}}\mathcal{S}^{\frac{2^*}{2}}\bigg)^{\frac{2}{2^*-2}}\leq\mathfrak{A}_u\leq\bigg(\frac{2N}{(N-2)bp}\bigg(\frac{2}{\lambda p}\bigg)^{\frac{2^*-p}{p-2}}\mathcal{S}^{-\frac{2^*}{2}}\bigg)^{\frac{2}{4-2^*}}.
\end{eqnarray}
Moreover, we also have
$\mathfrak{A}_u<\frac{2a}{(N-4)b}$ for $u\in\mathcal{M}^-$, $\mathfrak{A}_u=\frac{2a}{(N-4)b}$ for $u\in\mathcal{M}^0$ and $\mathfrak{A}_u>\frac{2a}{(N-4)b}$ for $u\in\mathcal{M}^+$.
\end{lemma}
\begin{proof}
By the H\"older and Sobolev inequalities, for every $u\in\h$, we have
\begin{eqnarray}\label{eqnew9100}
\mathfrak{B}_u\leq\mathcal{S}^{-\frac{2^*(p-2)}{2(2^*-2)}}\mathfrak{C}_u^{\frac{2^*-p}{2^*-2}}\mathfrak{A}_u^{\frac{2^*(p-2)}{2(2^*-2)}}.
\end{eqnarray}
Therefore, for $u\in\mathcal{M}$, we can see that
\begin{eqnarray*}
\frac\lambda2\mathfrak{C}_u&\leq&a\frac{N-2}{2N}\mathfrak{A}_u+b\frac{N-2}{2N}\mathfrak{A}_u^2+\frac\lambda2\mathfrak{C}_u\\
&=&\frac1p\mathfrak{B}_u\\
&\leq&\frac1p\mathcal{S}^{-\frac{2^*(p-2)}{2(2^*-2)}}\mathfrak{C}_u^{\frac{2^*-p}{2^*-2}}\mathfrak{A}_u^{\frac{2^*(p-2)}{2(2^*-2)}},
\end{eqnarray*}
which deduces that
\begin{eqnarray}\label{eq0014}
\frac{\lambda p}{2}\mathcal{S}^{\frac{2^*(p-2)}{2(2^*-2)}}\mathfrak{C}_u^{\frac{p-2}{2^*-2}}\leq\mathfrak{A}_u^{\frac{2^*(p-2)}{2(2^*-2)}}.
\end{eqnarray}
It follows from $u\in\mathcal{M}$ once more that
\begin{eqnarray}
a\frac{N-2}{2N}\mathfrak{A}_u&\leq& a\frac{N-2}{2N}\mathfrak{A}_u+b\frac{N-2}{2N}\mathfrak{A}_u^2\notag\\
&=&\frac1p\mathfrak{B}_u-\frac\lambda2\mathfrak{C}_u\notag\\
&\leq&\frac1p\mathcal{S}^{-\frac{2^*(p-2)}{2(2^*-2)}}\mathfrak{C}_u^{\frac{2^*-p}{2^*-2}}\mathfrak{A}_u^{\frac{2^*(p-2)}{2(2^*-2)}}-\frac\lambda2\mathfrak{C}_u\notag\\
&\leq&\frac1p\bigg(\frac{2}{\lambda p}\bigg)^{\frac{2^*-p}{p-2}}\mathcal{S}^{-\frac{2^*}{2}}\mathfrak{A}_u^{\frac{2^*}{2}}\label{eqnew0006}.
\end{eqnarray}
On the other hand, by \eqref{eq0014} and the fact that $u\in\mathcal{M}$ once more, we can see that
\begin{eqnarray}
b\frac{N-2}{2N}\mathfrak{A}_u^2&\leq& a\frac{N-2}{2N}\mathfrak{A}_u+b\frac{N-2}{2N}\mathfrak{A}_u^2\notag\\
&=&\frac1p\mathfrak{B}_u-\frac\lambda2\mathfrak{C}_u\notag\\
&\leq&\frac1p\mathcal{S}^{-\frac{2^*(p-2)}{2(2^*-2)}}\mathfrak{C}_u^{\frac{2^*-p}{2^*-2}}\mathfrak{A}_u^{\frac{2^*(p-2)}{2(2^*-2)}}-\frac\lambda2\mathfrak{C}_u\notag\\
&\leq&\frac1p\bigg(\frac{2}{\lambda p}\bigg)^{\frac{2^*-p}{p-2}}\mathcal{S}^{-\frac{2^*}{2}}\mathfrak{A}_u^{\frac{2^*}{2}}.\label{eq0017}
\end{eqnarray}
Since $2<2^*<4$ for $N\geq5$, \eqref{eq0015} follows immediately from \eqref{eqnew0006} and \eqref{eq0017}.  It remains to prove that $\mathfrak{A}_u<\frac{2a}{(N-4)b}$ for $u\in\mathcal{M}^-$, $\mathfrak{A}_u=\frac{2a}{(N-4)b}$ for $u\in\mathcal{M}^0$ and $\mathfrak{A}_u>\frac{2a}{(N-4)b}$ for $u\in\mathcal{M}^+$.  We only give the proof of the conclusion that $\mathfrak{A}_u<\frac{2a}{(N-4)b}$ for $u\in\mathcal{M}^-$, since the other two conclusions can be proved in a similar way.  Indeed, for every $u\in\mathcal{M}^-$, by the definition of $\mathcal{M}^-$ given by \eqref{eq0005}, we can see that
\begin{eqnarray*}
\frac{(N-2)(2N-3)b}{2N}\mathfrak{A}_u^2+\frac{(N-2)(N-1)a}{2N}\mathfrak{A}_u-(N+1)(\frac1p\mathfrak{B}_u-\frac{\lambda}{2}\mathfrak{C}_u)<0.
\end{eqnarray*}
Since $\mathcal{M}^-\subset\mathcal{M}$, we must have that
\begin{eqnarray*}
\frac1p\mathfrak{B}_u-\frac{\lambda}{2}\mathfrak{C}_u=a\frac{N-2}{2N}\mathfrak{A}_u+b\frac{N-2}{2N}\mathfrak{A}_u^2.
\end{eqnarray*}
It follows that
\begin{eqnarray*}
0&>&\frac{(N-2)(2N-3)b}{2N}\mathfrak{A}_u^2+\frac{(N-2)(N-1)a}{2N}\mathfrak{A}_u\\
&&-a\frac{(N-2)(N+1)}{2N}\mathfrak{A}_u-b\frac{(N-2)(N+1)}{2N}\mathfrak{A}_u^2\\
&=&\frac{N-2}{2N}\bigg((N-4)b\mathfrak{A}_u^2-2a\mathfrak{A}_u\bigg).
\end{eqnarray*}
Note that $N\geq5$.  Thus, we must have that $\mathfrak{A}_u<\frac{2a}{(N-4)b}$ for $u\in\mathcal{M}^-$, which completes the proof.
\end{proof}

Our third observation on the Pohozaev manifold $\mathcal{M}$ is the following.
\begin{lemma}\label{lem0004}
Let $u_0\in\mathcal{M}$ be a local minimum point of $\mathcal{E}(u)$ on $\mathcal{M}$ and $N\geq5$.  If $u_0\not\in\mathcal{M}^0$, then $\mathcal{E}'(u_0)=0$ in $H^{-1}(\bbr^N)$, where $H^{-1}(\bbr^N)$ is the dual space of $\h$ and $\mathcal{M}^0$ is given by \eqref{eq0006}.
\end{lemma}
\begin{proof}
The main idea of this proof comes from \cite{R06}, which was also used in \cite{A12,LY14}.  However, as we will see, since we need to deal with the high dimensions, we also need to borrow some ideas from \cite{HWW15}.  Suppose $u_0\in\mathcal{M}$ be a local minimum point of $\mathcal{E}(u)$ on $\mathcal{M}$.  Since $\Psi(u)$ is of $C^1$ in $\h$, by the method of Lagrange multipliers, there exists $\sigma\in\bbr$ such that $\mathcal{E}'(u_0)-\sigma\Psi'(u_0)=0$ in $H^{-1}(\bbr^N)$.  For the sake of clarity, we divide the following proof into two claims.

{\bf Claim~1}\quad We have $\Psi'(u_0)\not=0$ in $H^{-1}(\bbr^N)$.

Indeed, suppose the contrary that $\Psi'(u_0)=0$ in $H^{-1}(\bbr^N)$.  Then recalling the definition of $\Psi(u)$ given by \eqref{eq0003}, we can see that $u_0$ satisfies the following equation in the weak sense
\begin{eqnarray}\label{eq0018}
-\bigg(\frac{a(N-2)}{N}+\frac{2b(N-2)}{N}\mathfrak{A}_{u_0}\bigg)\Delta u_0+\lambda u_0=|u_0|^{p-2}u_0.
\end{eqnarray}
Thus, by the Pohozaev identity of \eqref{eq0018}, we have
\begin{eqnarray*}
\bigg(\frac{a(N-2)}{N}+\frac{2b(N-2)}{N}\mathfrak{A}_{u_0}\bigg)\frac{N-2}{2N}\mathfrak{A}_{u_0}+\frac\lambda 2\mathfrak{C}_{u_0}-\frac1p\mathfrak{B}_{u_0}=0,
\end{eqnarray*}
which together with the fact that $u_0\in\mathcal{M}$, implies
\begin{eqnarray*}
\bigg(\frac{a(N-2)}{N}+\frac{2b(N-2)}{N}\mathfrak{A}_{u_0}\bigg)\frac{N-2}{2N}\mathfrak{A}_{u_0}=\frac{a(N-2)}{2N}\mathfrak{A}_{u_0}+\frac{b(N-2)}{2N}\mathfrak{A}_{u_0}^2.
\end{eqnarray*}
It follows that
\begin{eqnarray}\label{eq0019}
\frac{N-2}{2N}\mathfrak{A}_{u_0}\bigg(\frac{b(N-4)}{N}\mathfrak{A}_{u_0}-\frac{2a}{N}\bigg)=0
\end{eqnarray}
Since $N\geq5$, by \eqref{eq0019}, we must have that $\mathfrak{A}_{u_0}=\frac{2a}{b(N-4)}$.  By Lemma~\ref{lem0003}, we must have that $u_0\in\mathcal{M}^0$, which is a contradiction.

{\bf Claim~2}\quad We have $\sigma=0$.

Indeed, suppose the contrary that $\sigma\not=0$.  Then recalling the definition of $\mathcal{E}(u)$ given by \eqref{eq0004} and the fact that $\mathcal{E}'(u_0)-\sigma\Psi'(u_0)=0$ in $H^{-1}(\bbr^N)$, we can see that $u_0$ satisfies the following equation in the weak sense
\begin{eqnarray}\label{eq0020}
&&-\bigg((1-\frac{(N-2)\sigma}{N})a+(1-\frac{2(N-2)\sigma}{N})b\mathfrak{A}_{u_0}\bigg)\Delta u_0+(1-\sigma)\lambda u_0\notag\\
&&=(1-\sigma)|u_0|^{p-2}u_0.
\end{eqnarray}
By the Pohozaev identity of \eqref{eq0020}, we have
\begin{eqnarray*}
&&\bigg((1-\frac{(N-2)\sigma}{N})a+(1-\frac{2(N-2)\sigma}{N})b\mathfrak{A}_{u_0}\bigg)\frac{N-2}{2N}\mathfrak{A}_{u_0}\\
&&+\frac{\lambda(1-\sigma)}{2}\mathfrak{C}_{u_0}-\frac{1-\sigma}{p}\mathfrak{B}_{u_0}=0,
\end{eqnarray*}
which together with the fact that $u_0\in\mathcal{M}$, implies
\begin{eqnarray*}
&&\bigg((1-\frac{(N-2)\sigma}{N})a+(1-\frac{2(N-2)\sigma}{N})b\mathfrak{A}_{u_0}\bigg)\frac{N-2}{2N}\mathfrak{A}_{u_0}\\
&&=\frac{a(N-2)(1-\sigma)}{2N}\mathfrak{A}_{u_0}+\frac{b(N-2)(1-\sigma)}{2N}\mathfrak{A}_{u_0}^2.
\end{eqnarray*}
It follows that
\begin{eqnarray}\label{eq0021}
\frac{(N-2)}{2N}\mathfrak{A}_{u_0}(\frac{b(N-4)}{N}\mathfrak{A}_{u_0}-\frac{2a}{N})\sigma=0.
\end{eqnarray}
Since $N\geq5$, by \eqref{eq0021}, we must have $\mathfrak{A}_{u_0}=\frac{2a}{b(N-4)}$, which together with Lemma~\ref{lem0003}, implies $u_0\in\mathcal{M}^0$.  It is also a contradiction.

Now, combining the above two claims and the fact that $\mathcal{E}'(u_0)-\sigma\Psi'(u_0)=0$ in $H^{-1}(\bbr^N)$, we must have $\mathcal{E}'(u_0)=0$ in $H^{-1}(\bbr^N)$.
\end{proof}

\subsection{Proof of Theorem~\ref{thm0001}}
We respectively denote $\inf_{\mathcal{M}}\mathcal{E}(u)$, $\inf_{\mathcal{M}^-}\mathcal{E}(u)$ and $\inf_{\mathcal{M}^+}\mathcal{E}(u)$ by $m$, $m^-$ and $m^+$.  Then we have the following.
\begin{lemma}\label{lem0005}
Let $N\geq5$.  If $\mathcal{B}_-\not=\emptyset$, then $m^\pm$ can be attained by some $u^\pm_{a,b,\lambda}$, which are both radial and nonnegative in $\bbr^N$.
\end{lemma}
\begin{proof}
Since $\mathcal{B}_-\not=\emptyset$, by Lemma~\ref{lem0001}, $\mathcal{M}^\pm\not=\emptyset$.  Let $\{u_n^\pm\}\subset\mathcal{M}^\pm$ respectively be a minimizing sequence of $\mathcal{E}(u)$ for $m^\pm$.  Then by the Schwartz symmetrization, there exists $\{u_n^{*,\pm}\}\subset H^1_r(\bbr^N)$ such that
\begin{eqnarray}\label{eq0026}
\mathfrak{A}_{u_n^{*,\pm}}\leq\mathfrak{A}_{u_n^\pm},\quad \mathfrak{B}_{u_n^{*,\pm}}=\mathfrak{B}_{u_n^\pm}\quad\text{and}\quad\mathfrak{C}_{u_n^{*,\pm}}=\mathfrak{C}_{u_n^\pm}.
\end{eqnarray}
Thus, by the definitions of $\mathcal{B}_-$ and $\mathcal{B}(u)$ respectively given by \eqref{eq0022} and \eqref{eq0023}, we must have that $\{u_n^{*,\pm}\}\subset\mathcal{B}_-$.  It follows from Lemma~\ref{lem0001} that there exist unique $0<t_{n,+}<t_{n,-}$ such that $v_n^\pm=u_n^{*,\pm}(t_{n,\pm}x)\in\mathcal{M}^\pm$.  Since \eqref{eq0026} holds, we must have from $\{u_n^\pm\}\subset\mathcal{M}^\pm$ that $F_{u_n^{*,\pm}}'(1)\geq0$.  It follows from Remark~\ref{rmk0001} that $t_{n,+}\leq1\leq t_{n,-}$, which together with Remark~\ref{rmk0001} once more, implies
\begin{eqnarray}
m^-+o_n(1)&=&F_{u_n^-}(1)\notag\\
&\geq&F_{u_n^-}(t_{n,-})\notag\\
&=&\frac{b}{4}\mathfrak{A}_{u_n^-}^2t_{n,-}^{4-2N}+\frac{a}{2}\mathfrak{A}_{u_n^-}t_{n,-}^{2-N}+(\frac\lambda 2\mathfrak{B}_{u_n^-}-\frac1p\mathfrak{C}_{u_n^-})t_{n,-}^{-N}\notag\\
&\geq&\frac{b}{4}\mathfrak{A}_{u_n^{*,-}}^2t_{n,-}^{4-2N}+\frac{a}{2}\mathfrak{A}_{u_n^{*,-}}t_{n,-}^{2-N}+(\frac\lambda 2\mathfrak{B}_{u_n^{*,-}}-\frac1p\mathfrak{C}_{u_n^{*,-}})t_{n,-}^{-N}\notag\\
&=&\frac{b}{4}\mathfrak{A}_{v_n^-}^2+\frac{a}{2}\mathfrak{A}_{v_n^-}+\frac\lambda 2\mathfrak{B}_{v_n^-}-\frac1p\mathfrak{C}_{v_n^-}\notag\\
&\geq&m^-\label{eq0027}
\end{eqnarray}
and
\begin{eqnarray}
m^++o_n(1)&=&F_{u_n^+}(1)\notag\\
&=&\frac{b}{4}\mathfrak{A}_{u_n^+}^2+\frac{a}{2}\mathfrak{A}_{u_n^+}+\frac\lambda 2\mathfrak{B}_{u_n^+}-\frac1p\mathfrak{C}_{u_n^+}\notag\\
&\geq&\frac{b}{4}\mathfrak{A}_{u_n^{*,+}}^2+\frac{a}{2}\mathfrak{A}_{u_n^{*,+}}+\frac\lambda 2\mathfrak{B}_{u_n^{*,+}}-\frac1p\mathfrak{C}_{u_n^{*,+}}\notag\\
&=&F_{u_{n}^{*,+}}(1)\notag\\
&\geq&F_{u_{n}^{*,+}}(t_{n,+})\notag\\
&\geq&m^+.\label{eq0028}
\end{eqnarray}
Therefore, $\{v_n^\pm\}$ are also minimizing sequences of $\mathcal{E}(u)$ for $m^\pm$, respectively.  By Lemma~\ref{lem0003} and \eqref{eq0014},  $\{v_n^\pm\}$ are bounded in $H^1_r(\bbr^N)$.  Therefore, without loss of generality, we may assume that $v_n^\pm=v_0^\pm+o_n(1)$ weakly in $H^1_r(\bbr^N)$.  Thanks to the Sobolev embedding theorem, we also have that $v_n^\pm=v_0^\pm+o_n(1)$ strongly in $L^q(\bbr^N)(2\leq q<2^*)$.  Clearly, $v_0^\pm\in\mathcal{B}_-$, which together with Lemma~\ref{lem0001}, implies that there exist unique $0<t_{0,+}<t_{0,-}$ such that $v_0^{*,\pm}=v_0^\pm(t_{0,\pm}x)\in\mathcal{M}^\pm$.  Since $v_n^\pm=v_0^\pm+o_n(1)$ weakly in $H^1_r(\bbr^N)$, we must have that $F_{v_0^{\pm}}'(1)\geq0$.  It follows from Remark~\ref{rmk0001} that $t_{0,+}\leq1\leq t_{0,-}$.  Now, by similar arguments as used for \eqref{eq0027} and \eqref{eq0028}, we can see that $m^\pm$ can be attained by $v_0^{*,\pm}$.  Note that $|v_0^{*,\pm}|$ also attain $m^\pm$ by the definitions of $\mathcal{M}^\pm$, respectively.  Thus, $m^\pm$ can be attained by some $u^\pm_{a,b,\lambda}$, which are both radial and nonnegative in $\bbr^N$.
\end{proof}

Now, we can give the proof of Theorem~\ref{thm0001}.

\noindent\textbf{Proof of Theorem~\ref{thm0001}.}\quad $(1)$\quad Since $N\geq5$, if $\mathcal{F}_{a,b}(U_\lambda)<1$, then by the fact that $U_\lambda\in\mathcal{M}^*$, we can see that $U_\lambda\in\mathcal{B}_-$.  It follows from Lemma~\ref{lem0005} that $m^\pm$ can be attained by some $u^\pm_{a,b,\lambda}$, which are both radial and nonnegative in $\bbr^N$.  By a direct calculation, we can see that the function $H(s)=\frac{a}{N}s+\frac{(4-N)b}{4N}s^2$ is strictly increasing on $(0, \frac{2a}{(N-4)b})$ and strictly decreasing on $(\frac{2a}{(N-4)b}, +\infty)$.  Note that $\mathcal{E}(u)=\frac{a}{N}\mathfrak{A}_{u}+\frac{(4-N)b}{4N}\mathfrak{A}_{u}^2$ for all $u\in\mathcal{M}$.  Thus, by Lemma~\ref{lem0003}, $u_{a,b,\lambda}^\pm$ are both local minimum points of $\mathcal{E}(u)$ on $\mathcal{M}$.  Moreover, by Remark~\ref{rmk0001}, we also have that $m^+=m$.
Thanks to Lemma~\ref{lem0004} and the maximum principle, $u_{a,b,\lambda}^\pm$ are two radial positive solutions to the autonomous form of $(\mathcal{P}_{a,b})$.

$(2)$\quad If $\mathcal{F}_{a,b}(U_\lambda)=1$, then by the fact that $U_\lambda\in\mathcal{M}^*$, we can see that $U_\lambda\in\mathcal{B}_0$, where $\mathcal{B}_0$ is given by \eqref{eqnew0011}.  Consider the function $v_t(x)=U_\lambda(tx)$, where $t=\bigg(\frac{(N-4)b\mathfrak{A}_{U_\lambda}}{2a}\bigg)^{\frac{1}{N-2}}$.  Then by a direct calculation, we can see that the following equation holds in the weak sense
\begin{eqnarray*}
-\bigg(a+b\mathfrak{A}_{v_t}\bigg)\Delta v_t&=&-\bigg(a+b\mathfrak{A}_{U_\lambda}t^{2-N}\bigg)t^{2-N}\Delta U_\lambda\\
&=&\bigg(a+b\mathfrak{A}_{U_\lambda}t^{2-N}\bigg)t^{2-N}(U_\lambda^{p-1}-\lambda U_\lambda)\\
&=&\bigg(a+b\mathfrak{A}_{U_\lambda}t^{2-N}\bigg)t^{2}(v_t^{p-1}-\lambda v_t)\\
&=&\frac{(N-2)a}{N-4}t^2(v_t^{p-1}-\lambda v_t)\\
&=&v_t^{p-1}-\lambda v_t.
\end{eqnarray*}
Thus, $v_t(x)$ is a radial positive solution to the autonomous form of $(\mathcal{P}_{a,b})$.  Suppose $\mathcal{M}^+\cup\mathcal{M}^-\not=\emptyset$.  Then by Lemma~\ref{lem0001}, there exists $u\in\mathcal{B}_-\cap\mathcal{C}$.  It follows from Lemma~\ref{lem0002} that there exists $t=\bigg(\frac{\frac1p\mathfrak{B}_{u}-\frac\lambda2\mathfrak{C}_{u}}{\frac{N-2}{2N}\mathfrak{A}_{u}}\bigg)^{\frac12}$ such that $u_t\in\mathcal{M}^*$.   By a similar argument as used in the proof of Lemma~\ref{lem0005}, we can see that $I(U_\lambda)=\inf_{\mathcal{M}^*}I(u)$.
Note that $I(u)=\frac{1}{N}\mathfrak{A}_{u}$ for $u\in\mathcal{M}^*$.  Therefore, we must have
\begin{eqnarray}\label{eq0031}
\mathfrak{A}_{U_\lambda}=\min\{\mathfrak{A}_{u}\mid u\in \mathcal{M}^*\},
\end{eqnarray}
which together with a similar argument as used in \eqref{eq0053} and the fact that $u\in\mathcal{B}_-$, implies $U_\lambda\in\mathcal{B}_-$.  It is impossible since we already have $U_\lambda\in\mathcal{B}_0$.  Thus, thanks to Lemma~\ref{lem0001}, if $\mathcal{F}_{a,b}(U_\lambda)=1$, then $\mathcal{M}=\mathcal{M}^0$.

$(3)$\quad Suppose the autonomous form of $(\mathcal{P}_{a,b})$ has a solution $u$ if $\mathcal{F}_{a,b}(U_\lambda)>1$.  Then by the fact that $U_\lambda\in\mathcal{M}^*$, we can see that $U_\lambda\in\mathcal{B}_+$, where $\mathcal{B}_+$ is given by \eqref{eqnew0012}.  It follows that $(\mathcal{P}_{a,b,\lambda})$ has a solution $u$ if $U_\lambda\in\mathcal{B}_+$.  Note that we must have that $u\in\mathcal{M}\cap\mathcal{C}$.  Thus, by Lemmas~\ref{lem0002} and \ref{lem0001}, we can see that $\mathcal{B}(u)\leq1$ and there exists $t=\bigg(\frac{\frac1p\mathfrak{B}_{u}-\frac\lambda2\mathfrak{C}_{u}}{\frac{N-2}{2N}\mathfrak{A}_{u}}\bigg)^{\frac12}$ such that $u_t\in\mathcal{M}^*$.  By a direct calculation, we can see that
\begin{eqnarray*}
1\geq\mathcal{B}(u)=\mathcal{F}_{a,b}(u_t)\geq\mathcal{F}_{a,b}(U_\lambda),
\end{eqnarray*}
which together with \eqref{eq0031}, implies that $U_\lambda\in\mathcal{B}_-\cup\mathcal{B}_0$.  It is impossible since we have $U_\lambda\in\mathcal{B}_+$.
\qquad\raisebox{-0.5mm}{%
\rule{1.5mm}{4mm}}\vspace{6pt}

\subsection{The scaling technique}
In this section, we will study $(\mathcal{P}_{a,b,\lambda})$ by the scaling technique.  Our main observation is the following two lemmas.
\begin{lemma}\label{lem0006}
Let $a,b,\lambda>0$, $N\geq3$ and $2<p<2^*$.  Suppose $u$ is a solution of $(\mathcal{P}_{a,b,\lambda})$, then there exist $s,t$ and $\alpha>0$ such that $U_{\alpha}(x)=su(tx)$ up to a translation.
\end{lemma}
\begin{proof}
Let $v(x)=su(tx)$, where $s,t>0$ are constants.  Then for every $\varphi\in \h$, we have
\begin{eqnarray*}
&&\int_{\bbr^N}\nabla v(x)\nabla \varphi(x)dx\\
&=&\int_{\bbr^N}\nabla su(tx)\nabla \varphi(x)dx\\
&=&t^2s\int_{\bbr^N}\nabla u(x)\nabla \varphi(\frac{x}{t})d\frac{x}{t}\\
&=&\frac{t^{2-N}s}{(a+b\mathfrak{A}_{u})}\int_{\bbr^N}(u^{p-1}(x)-\lambda u(x))\varphi(\frac{x}{t})dx\\
&=&\frac{t^{2}s}{(a+b\mathfrak{A}_{u})}\int_{\bbr^N}(u^{p-1}(x)-\lambda u(x))\varphi(\frac{x}{t})d\frac{x}{t}\\
&=&\frac{t^{2}}{(a+b\mathfrak{A}_{u})}\int_{\bbr^N}(s^{2-p}v^{p-1}(x)-\lambda v(x))\varphi(x)dx.
\end{eqnarray*}
Since $a,b>0$, for fixed $s>0$, the equation $s^{2-p}t^2=a+b\mathfrak{A}_{u}$ must have a unique solution $t>0$.  Let $\alpha=\frac{\lambda t^2}{a+b\mathfrak{A}_{u}}$.  Then by the uniqueness of $U_\alpha$, we can see that $U_{\alpha}(x)=su(tx)$ up to a translation.
\end{proof}

\begin{lemma}\label{lemnew0001}
Let $a,b,\lambda>0$, $N\geq3$ and $2<p<2^*$.  Then $\bigg(\frac{\lambda}{\alpha}\bigg)^{\frac{1}{p-2}}U_\alpha(tx)$ is a positive solution to the autonomous form of $(\mathcal{P}_{a,b})$ if and only if there exist two positive constants $t$ and $\alpha$ satisfying $\bigg(\frac{\alpha}{\lambda}\bigg)^{\frac12}t=\gamma>0$ and $\gamma$ is the solution of \eqref{eqnew0005}.
\end{lemma}
\begin{proof}
Let $u(x)=sU_{\alpha}(tx)$, then for every $\varphi\in\h$, we have
\begin{eqnarray*}
&&(a+b\mathfrak{A}_{u})\int_{\bbr^N}\nabla u(x)\nabla \varphi(x)dx\\
&=&(a+bt^{2-N}s^2\mathfrak{A}_{U_{\alpha}})\int_{\bbr^N}\nabla u(x)\nabla \varphi(x)dx\\
&=&(a+bt^{2-N}s^2\mathfrak{A}_{U_{\alpha}})t^{2-N}\int_{\bbr^N}\nabla sU_{\alpha}(x)\nabla \varphi(\frac{x}{t})dx\\
&=&(a+bt^{2-N}s^2\mathfrak{A}_{U_{\alpha}})t^{2-N}s\int_{\bbr^N}(U_{\alpha}^{p-1}(x)-\alpha U_{\alpha}(x))\varphi(\frac{x}{t})dx\\
&=&(a+bt^{2-N}s^2\mathfrak{A}_{U_{\alpha}})t^{2}\int_{\bbr^N}(s^{2-p}u^{p-1}(x)-\alpha u(x))\varphi(x)dx.
\end{eqnarray*}
Thus, $u(x)$ is a solution to the autonomous form of $(\mathcal{P}_{a,b})$ if and only if
$$
\left\{\aligned &1=(a+bt^{2-N}s^2\mathfrak{A}_{U_{\alpha}})t^{2}s^{2-p},\\
&\lambda=\alpha(a+bt^{2-N}s^2\mathfrak{A}_{U_{\alpha}})t^{2},
\endaligned
\right.
$$
which is equivalent to $s=\bigg(\frac{\lambda}{\alpha}\bigg)^{\frac{1}{p-2}}$ and $t$ is a solution of the following equation
\begin{eqnarray}\label{eq9956}
\frac{\lambda}{\alpha}=a t^2+b(\frac{\lambda}{\alpha})^{\frac{2}{p-2}}\mathfrak{A}_{U_{\alpha}}t^{4-N}.
\end{eqnarray}
Thanks to \eqref{eqnew0010}, we can see that $\mathfrak{A}_{U_{\alpha}}=\bigg(\frac{\alpha}{\lambda}\bigg)^{\frac{2}{p-2}-\frac{N-2}{2}}\mathfrak{A}_{U_{\lambda}}$.  It follows from \eqref{eq9956} that
\begin{eqnarray*}
1&=&a\bigg(\bigg(\frac{\alpha}{\lambda}\bigg)^{\frac12}t\bigg)^2+b\mathfrak{A}_{U_{\lambda}}\bigg(\frac{\alpha}{\lambda}\bigg)^{1-\frac{N-2}{2}}t^{4-N}\\
&=&a\bigg(\bigg(\frac{\alpha}{\lambda}\bigg)^{\frac12}t\bigg)^2+b\bigg(\bigg(\frac{\alpha}{\lambda}\bigg)^{\frac12}t\bigg)^{4-N}\mathfrak{A}_{U_{\lambda}}.
\end{eqnarray*}
Thus, let $\bigg(\frac{\alpha}{\lambda}\bigg)^{\frac12}t=\gamma$, then $\gamma$ is the solution of \eqref{eqnew0005} if and only if $\bigg(\frac{\lambda}{\alpha}\bigg)^{\frac{1}{p-2}}U_\alpha(tx)$ is a positive solution to the autonomous form of $(\mathcal{P}_{a,b})$.
\end{proof}

Due to Lemmas~\ref{lem0006} and \ref{lemnew0001}, we can give a proof of Theorem~\ref{thm0002}.

\noindent\textbf{Proof of Theorem~\ref{thm0002}.}\quad By Lemmas~\ref{lem0006} and \ref{lemnew0001}, we can see that the solution to the autonomous form of $(\mathcal{P}_{a,b})$ must be of the form $\bigg(\frac{\lambda}{\alpha}\bigg)^{\frac{1}{p-2}}U_\alpha(tx)$, where $U_\alpha$ is the unique positive radial solution of \eqref{eqnew1006} for $\lambda=\alpha$, $\alpha$ and $t$ are two positive constants satisfying $\bigg(\frac{\alpha}{\lambda}\bigg)^{\frac12}t=\gamma>0$ and $\gamma$ is the solution of \eqref{eqnew0005}.
Moreover, due to Lemma~\ref{lemnew0001}, the number of positive solutions to $(\mathcal{P}_{a,b,\lambda})$ equals to the number of solutions to the equation \eqref{eqnew0005}.
\qquad\raisebox{-0.5mm}{%
\rule{1.5mm}{4mm}}\vspace{6pt}

We close this section by

\noindent\textbf{Proof of Theorem~\ref{thm0004}.}\quad By making a direct observation on the equation \eqref{eqnew0005}, we can see that \eqref{eqnew0005} has a unique solution for all $a,b,\lambda>0$ in the case $N=3$; \eqref{eqnew0005} has a unique solution for $b\mathfrak{A}_{U_{\lambda}}<1$ and has no solution for $b\mathfrak{A}_{U_{\lambda}}\geq1$ in the case $N=4$; \eqref{eqnew0005} has exact two solutions for $\mathcal{F}_{a,b}(U_\lambda)<1$, has a unique solution for $\mathcal{F}_{a,b}(U_\lambda)=1$ and has no solution for $\mathcal{F}_{a,b}(U_\lambda)>1$ in the cases $N\geq5$.  Thus, the conclusion follows immediately from Theorems~\ref{thm0001} and \ref{thm0002} and the results in \cite{A12,A13}.
\qquad\raisebox{-0.5mm}{%
\rule{1.5mm}{4mm}}\vspace{6pt}

\section{The non-autonomous case}

\subsection{The Nehari manifold $\mathcal{N}_V$ for $p\in(2, 4]\cap(2, 2^*)$}
We first consider the cases $p\in(2, 4)\cap(2,2^*)$.  Let
\begin{eqnarray}
\mathcal{D}&=&\{u\in\h\backslash\{0\}\mid \mathcal{G}(u)<0\}.\label{eqnew9001}
\end{eqnarray}
where
\begin{eqnarray*}
\mathcal{G}(u)=a\mathfrak{A}_u+\mathfrak{C}_{u,V}-\frac{4-p}{2}\bigg(\frac{(p-2)\mathfrak{B}_u}{2b\mathfrak{A}_u^2}\bigg)^{\frac{p-2}{4-p}}\mathfrak{B}_u.
\end{eqnarray*}
Then our first observation on the Nehari manifold by the fibering map $G_u(t)$ is the following.
\begin{lemma}\label{lemnew0002}
Let $p\in(2, 4)\cap(2,2^*)$ and the condition $(V)$ hold.  Then there exist unique $0<t^-<t^+<+\infty$ such that $t^\pm u\in\mathcal{N}_V^\pm$ for $u\in\mathcal{D}$, where $\mathcal{N}_V^\pm$ are given by \eqref{eq1005} and \eqref{eq1007}.  Moreover, $G_u(s)$ is strictly increasing on $(0, t^-)$, strictly decreasing on $(t^-, t^+)$ and strictly increasing on $(t^+, +\infty)$.
\end{lemma}
\begin{proof}
By a direct calculation, we can see that $G'_u(t)=t(b\mathfrak{A}_u^2t^2-\mathfrak{B}_ut^{p-2}+a\mathfrak{A}_u+\mathfrak{C}_{u,V})$.  Set
\begin{eqnarray*}\label{eqnew9002}
g_u(t)=b\mathfrak{A}_u^2t^2-\mathfrak{B}_ut^{p-2}+a\mathfrak{A}_u+\mathfrak{C}_{u,V}.
\end{eqnarray*}
Then by the condition $(V)$, we can see that $g_u(t)$ is strictly decreasing on $(0, t^*)$ and strictly increasing on $(t^*, +\infty)$, where
\begin{eqnarray}\label{eqnew9007}
t^*=\bigg(\frac{(p-2)\mathfrak{B}_u}{2b\mathfrak{A}_u^2}\bigg)^{\frac{1}{4-p}}.
\end{eqnarray}
It follows that $g_u(t^*)=\min_{t\geq0}g_u(t)=\mathcal{G}(u)$.  Now, the conclusion follows immediately from the definition of $\mathcal{D}$.
\end{proof}

Next, we consider the case $p=4<2^*$, which implies $N=3$.  We define
\begin{eqnarray}\label{eqnew9003}
\mathcal{Q}=\{u\in\h\backslash\{0\}\mid b\mathfrak{A}_u^2-\mathfrak{B}_u<0\}.
\end{eqnarray}
Then we have the following.
\begin{lemma}\label{lemnew0003}
Let $N=3$, $p=4$ and the condition $(V)$ hold.  Then there exists unique $0<t^0<+\infty$ such that $t^0 u\in\mathcal{N}_V$ for $u\in\mathcal{Q}$, where $\mathcal{N}_V$ is given by \eqref{eq0102}.  Moreover, $G_u(s)$ is strictly increasing on $(0, t^0)$ and strictly decreasing on $(t^0, +\infty)$.
\end{lemma}
\begin{proof}
The proof is similar to that of Lemma~\ref{lemnew0002}.
\end{proof}

\begin{remark}\label{rmknew0001}
By checking the proof of Lemmas~\ref{lemnew0002} and \ref{lemnew0003}, we can see that $tu\not\in\mathcal{N}_V^\pm$ for all $t\geq0$ if $u\not\in\mathcal{D}$ in the cases $p\in(2, 4)\cap(2,2^*)$  and $tu\not\in\mathcal{N}_V$ for all $t\geq0$ if $u\not\in\mathcal{Q}$ in the case $N=3$ and $p=4$.
\end{remark}

Let
\begin{eqnarray}\label{eqnew9006}
\mathcal{S}_{p,a,V}=\inf_{u\in\h\backslash\{0\}}\frac{a\mathfrak{A}_u+\mathfrak{C}_{u,V}}{\mathfrak{B}_u^{\frac2p}}.
\end{eqnarray}
Then by the condition $(V)$, we can see that $\mathcal{S}_{p,a,V}>0$ is well defined.
\begin{lemma}\label{lemnew0004}
Let $a>0$ and $p\in(2, 4]\cap(2,2^*)$.  If the condition $(V)$ holds, then there exists $b_*(a)>0$ such that $\mathcal{D}$ and $\mathcal{Q}$ are both nonempty sets for $0<b<b_*(a)$, where $\mathcal{D}$ and $\mathcal{Q}$ are respectively given by \eqref{eqnew9001} and \eqref{eqnew9003}.
\end{lemma}
\begin{proof}
Let ${u_n}$ be a minimizing sequence of $\mathcal{S}_{p,a,V}$.  Then for $p\in(2, 4)\cap(2,2^*)$, we have from the condition $(V)$ that
\begin{eqnarray*}
\mathcal{G}(u_n)&\leq&a\mathfrak{A}_{u_n}+\mathfrak{C}_{u_n,V}
-\frac{(4-p)a^{\frac{2(p-2)}{4-p}}}{2}\bigg(\frac{(p-2)}{2b}\bigg)^{\frac{p-2}{4-p}}\frac{\mathfrak{B}_{u_n}^{\frac{2}{4-p}}}
{(a\mathfrak{A}_{u_n}+\mathfrak{C}_{u_n,V})^{\frac{2(p-2)}{4-p}}}\\
&=&\frac{\mathfrak{B}_{u_n}^{\frac{2}{4-p}}}{(a\mathfrak{A}_{u_n}+\mathfrak{C}_{u_n,V})^{\frac{2(p-2)}{4-p}}}
\bigg(\mathcal{S}_{p,a,V}^{\frac{p}{4-p}}
-\frac{(4-p)a^{\frac{2(p-2)}{4-p}}}{2}\bigg(\frac{(p-2)}{2b}\bigg)^{\frac{p-2}{4-p}}+o_n(1)\bigg).\\
\end{eqnarray*}
For $N=3$ and $p=4$, we also have from the condition $(V)$ that
\begin{eqnarray*}
b\mathfrak{A}_{u_n}^2-\mathfrak{B}_{u_n}&\leq&\frac{b}{a^2}(a\mathfrak{A}_{u_n}+\mathfrak{C}_{u_n,V})^2-\mathfrak{B}_{u_n}\\
&=&\mathfrak{B}_{u_n}(\frac{b}{a^2}\mathcal{S}_{4,a,V}^2+o_n(1)-1).
\end{eqnarray*}
Thus, there exists $b_*(a)>0$ such that $\mathcal{D}$ and $\mathcal{Q}$ are both nonempty sets for $0<b<b_*(a)$.
\end{proof}

By Lemmas~\ref{lemnew0002}--\ref{lemnew0004}, we can see that $\mathcal{N}_V^-\not=\emptyset$ for $0<b<b_*(a)$ in all the cases $p\in(2, 4]\cap(2, 2^*)$ and $\mathcal{N}_V^+\not=\emptyset$ for $0<b<b_*(a)$ in the cases $p\in(2, 4)\cap(2, 2^*)$.  Thus, $m^\pm=\inf_{\mathcal{N}_V^\pm}\mathcal{J}_V(u)$ are both well defined respectively in these cases, where $\mathcal{J}_V(u)$ is given by \eqref{eq0104}.
\begin{lemma}\label{lemnew0005}
Let $a>0$, $p\in(2, 4]\cap(2, 2^*)$ and the condition $(V)$ hold.  Then $m^->0$ for $0<b<b_*(a)$.  Moreover, $m^+<0$ for $0<b<b_*(a)$ in the cases $p\in(2, 4)\cap(2, 2^*)$.
\end{lemma}
\begin{proof}
Let $u\in\mathcal{N}_V^-$.  Then we have
\begin{eqnarray}
b\mathfrak{A}_u^2-\mathfrak{B}_u+a\mathfrak{A}_u+\mathfrak{C}_{u,V}=0\label{eqnew9004}\\
3b\mathfrak{A}_u^2-(p-1)\mathfrak{B}_u+a\mathfrak{A}_u+\mathfrak{C}_{u,V}<0.\label{eqnew9005}
\end{eqnarray}
By \eqref{eqnew9100}, \eqref{eqnew9004} and the condition $(V)$, we can see that
\begin{eqnarray*}
a\mathfrak{A}_u+\mathfrak{C}_{u,V}&\leq&\mathfrak{B}_u\\
&\leq&C_1\mathfrak{C}_u^{\frac{2^*-p}{2^*-2}}\mathfrak{A}_u^{\frac{2^*(p-2)}{2(2^*-2)}}\\
&\leq&C_2\mathfrak{C}_{u,V}^{\frac{2^*-p}{2^*-2}}\mathfrak{A}_u^{\frac{2^*(p-2)}{2(2^*-2)}}
\end{eqnarray*}
It follows from the Young inequality that
\begin{eqnarray}\label{eqnew9017}
a\mathfrak{A}_u\leq C_2\mathfrak{C}_{u,V}^{\frac{2^*-p}{2^*-2}}\mathfrak{A}_u^{\frac{2^*(p-2)}{2(2^*-2)}}-\mathfrak{C}_{u,V}
\leq C_3\mathfrak{A}_u^{\frac{2^*}{2}},
\end{eqnarray}
which implies $\mathfrak{A}_u\geq C_4$.
On the other hand, by \eqref{eqnew9004} and \eqref{eqnew9005}, we must have that $(4-p)\mathfrak{A}_u^2<(p-2)(a\mathfrak{A}_u+\mathfrak{C}_{u,V})$.  Thus, for $p\in(2, 4)\cap(2, 2^*)$, we can see that
\begin{eqnarray}
\mathcal{J}_V(u)&=&\frac{p-2}{2p}(a\mathfrak{A}_u+\mathfrak{C}_{u,V})-\frac{4-p}{4p}b\mathfrak{A}_u^2\notag\\
&>&\frac{4-p}{4p}b\mathfrak{A}_u^2\label{eqnew9025}\\
&\geq&C_5.\notag
\end{eqnarray}
For the case $p=4$, by a similar argument as used for \eqref{eqnew9025}, we can see that $\mathcal{J}_V(u)\geq\frac{(p-2)a}{2p}\mathfrak{A}_u\geq C_6$.
Since $u\in\mathcal{N}_V^-$ is arbitrary, by Lemmas~\ref{lemnew0002}--\ref{lemnew0004}, we must have that $m^->0$ for $0<b<b_*(a)$ in all the cases $p\in(2, 4]\cap(2, 2^*)$.  Next, we prove that $m^+<0$ for $0<b<b_*(a)$ in the cases $p\in(2, 4)\cap(2, 2^*)$.  Indeed, by choosing $b_*(a)$ small enough if necessary, we can see that
\begin{eqnarray*}
\mathcal{S}_{p,a,V}<a^{\frac{2(p-2)}{p}}\bigg(\frac{(4-p)(p+2)}{4p}\bigg)^{\frac{4-p}{p}}\bigg(\frac{(p-2)}{2b}\bigg)^{\frac{p-2}{p}}
\end{eqnarray*}
for $0<b<b_*(a), $where $\mathcal{S}_{p,a,V}$ is given by \eqref{eqnew9006}.  It follows that there exists $u\in\h\backslash\{0\}$ such that
\begin{eqnarray}\label{eqnew9011}
a\mathfrak{A}_u+\mathfrak{C}_{u,V}<a^{\frac{2(p-2)}{p}}\bigg(\frac{(4-p)(p+2)}{4p}\bigg)^{\frac{4-p}{p}}\bigg(\frac{(p-2)}{2b}\bigg)^{\frac{p-2}{p}}\mathfrak{B}_u^{\frac2p}.
\end{eqnarray}
Since $2<p<4$, $\frac{p+2}{2p}<1$.  Thus, we must have that $u\in\mathcal{D}$.  By Lemma~\ref{lemnew0002}, there exists $t^+>0$ such that $t^+u\in\mathcal{N}_V^+$.  Thanks to Lemma~\ref{lemnew0002} once more, we also have that $\mathcal{J}_V(t^+u)\leq\mathcal{J}_V(t^*u)$, where $t^*$ is given by \eqref{eqnew9007}.  It follows from \eqref{eqnew9011} that
\begin{eqnarray*}
\mathcal{J}_V(t^+u)&\leq&\mathcal{J}_V(t^*u)\\
&=&\frac{(t^*)^2}{2}\bigg(a\mathfrak{A}_u+\mathfrak{C}_{u,V}-\frac{(4-p)(p+2)(t^*)^{p-2}}{4p}\mathfrak{B}_u\bigg)\\
&\leq&\frac{(t^*)^2}{2}\bigg(a\mathfrak{A}_u+\mathfrak{C}_{u,V}\\
&&-\frac{(4-p)(p+2)a^{\frac{2(p-2)}{4-p}}}{4p}\bigg(\frac{(p-2)\mathfrak{B}_u}{2b(a\mathfrak{A}_u+\mathfrak{C}_{u,V})^2}\bigg)^{\frac{p-2}{4-p}}\mathfrak{B}_u\bigg)\\
&=&\frac{(t^*)^2}{2(a\mathfrak{A}_u+\mathfrak{C}_{u,V})^{\frac{2(p-2)}{4-p}}}\bigg((a\mathfrak{A}_u+\mathfrak{C}_{u,V})^{\frac{p}{4-p}}\\
&&-\frac{(4-p)(p+2)a^{\frac{2(p-2)}{4-p}}}{4p}\bigg(\frac{(p-2)}{2b}\bigg)^{\frac{p-2}{4-p}}\mathfrak{B}_u^{\frac{2}{4-p}}\bigg)\\
&<&0,
\end{eqnarray*}
which implies $m^+<0$ for $0<b<b_*(a)$ in the cases $p\in(2, 4)\cap(2, 2^*)$.
\end{proof}

We close this section by the following observation on $\mathcal{N}_V$ for $0<b<b_*(a)$ in all the cases $p\in(2, 4]\cap(2, 2^*)$.
\begin{lemma}\label{lemnew0006}
Let $a>0$ and the condition $(V)$ holds.  If $u_0\in\mathcal{N}_V^-$ minimizes $\mathcal{J}_V(u)$ on $\mathcal{N}_V^-$ for $0<b<b_*(a)$ in the cases $p\in(2, 4]\cap(2, 2^*)$, then $u_0$ is also a critical point of $\mathcal{J}_V(u)$ in $\h$.
\end{lemma}
\begin{proof}
Let $u_0\in\mathcal{N}_V^-$ be a minimum point of $\mathcal{J}_V(u)$ on $\mathcal{N}_V^-$ for $0<b<b_*(a)$ in the cases $p\in(2, 4]\cap(2, 2^*)$.  by the definition of $\mathcal{N}_V^-$, $u_0$ is also a minimum point of $\mathcal{J}_V(u)$ on $\mathcal{N}_V$ for $0<b<b_*(a)$.  Thanks to the method of Lagrange multipliers, there exists $\tau\in\bbr$ such that $\mathcal{J}'_V(u_0)-\tau\Phi'_V(u_0)=0$ in $H^{-1}(\bbr^N)$, where $\Phi_V(u)=\mathcal{J}'_V(u)u$.  It follows from $u_0\in\mathcal{N}_V$ and a direct calculation that $0=\mathcal{J}'_V(u_0)u_0=\tau\Phi'_V(u_0)u_0=\tau\mathcal{G}_{u_0}''(1)$.  Note that $u_0\in\mathcal{N}_V^-$, thus, $\mathcal{G}_{u_0}''(1)<0$, which implies $\tau=0$.  It follows that $\mathcal{J}'_V(u_0)=0$ in $H^{-1}(\bbr^N)$.
\end{proof}
\subsection{A local compactness result}
Let
\begin{eqnarray*}\label{eqnew9018}
\mathcal{J}_\infty(u)=\frac{b}{4}\mathfrak{A}_u^2+\frac{a}{2}\mathfrak{A}_u+\frac{v_\infty}{2}\mathfrak{C}_u-\frac1p\mathfrak{B}_u
\end{eqnarray*}
be the corresponding functional of the following equation
\begin{eqnarray}\label{eqnew9019}
\left\{\aligned&-\bigg(a+b\int_{\bbr^N}|\nabla u|^2dx\bigg)\Delta u+v_\infty u=|u|^{p-2}u&\text{ in }\bbr^N,\\%
&u\in\h,\endaligned\right.
\end{eqnarray}
Then $\mathcal{J}_\infty(u)$ and \eqref{eqnew9019} can respectively be seen as the ``limit'' functional and equation of $\mathcal{J}_V(u)$ and $(\mathcal{P}_{a,b})$.
Let $\mathcal{D}_\infty=\{u\in\h\backslash\{0\}\mid\mathcal{G}_\infty(u)<0\}$, where
\begin{eqnarray*}
\mathcal{G}_\infty(u)=a\mathfrak{A}_u+v_\infty\mathfrak{C}_{u}-\frac{4-p}{2}\bigg(\frac{(p-2)\mathfrak{B}_u}{2b\mathfrak{A}_u^2}\bigg)^{\frac{p-2}{4-p}}\mathfrak{B}_u.
\end{eqnarray*}
Let
\begin{eqnarray}\label{eqnew9020}
\mathcal{N}_\infty=\{u\in\h\backslash\{0\}\mid \mathcal{J}_\infty'(u)u=0\}.
\end{eqnarray}
Then it is easy to see that all nontrivial critical points of $\mathcal{J}_\infty(u)$ are contained in $\mathcal{N}_\infty$.  Let
\begin{eqnarray*}
G_{u,\infty}(t)&=&\mathcal{J}_\infty(tu)\\
&=&\frac{bt^4}{4}\mathfrak{A}_u^2+\frac{at^2}{2}\mathfrak{A}_u+\frac{t^2v_\infty}2\mathfrak{C}_u^2-\frac{t^p}{p}\mathfrak{B}_u.
\end{eqnarray*}
Then by a direct calculation, we can see that $G_{u,\infty}(t)$ is of $C^2$ in $\bbr^+$ for every $u\in\h$ and $G_{u,\infty}'(t)=0$ if and only if $tu\in\mathcal{N}_\infty$.  Thus, it is natural to divide the Nehari type manifold $\mathcal{N}_{\infty}$ into the following three parts:
\begin{eqnarray}
\mathcal{N}_\infty^-&=&\{u\in\mathcal{N}\mid G_{u,\infty}''(1)<0\};\label{eqnew9021}\\
\mathcal{N}_\infty^0&=&\{u\in\mathcal{N}\mid G_{u,\infty}''(1)=0\};\label{eqnew9022}\\
\mathcal{N}_\infty^+&=&\{u\in\mathcal{N}\mid G_{u,\infty}''(1)>0\}.\label{eqnew9023}.
\end{eqnarray}
Then choosing $b_*(a)$ small enough if necessary and by similar arguments as used in Lemmas~\ref{lemnew0002}--\ref{lemnew0006}, we can obtain the following.
\begin{lemma}\label{lemnew0007}
Let $a>0$, $0<b<b_*(a)$ and $p\in(2, 4]\cap(2, 2^*)$.  Then we have the following.
\begin{enumerate}
\item[$(1)$] $\mathcal{D}_\infty$ and $\mathcal{Q}$ are both nonempty sets.
\item[$(2)$] If $p\in(2, 4)\cap(2,2^*)$.  Then there exist unique $0<t_\infty^-<t_\infty^+<+\infty$ such that $t_\infty^\pm u\in\mathcal{N}_\infty^\pm$ for $u\in\mathcal{D}_\infty$, where $\mathcal{N}_\infty^\pm$ are given by \eqref{eqnew9021} and \eqref{eqnew9023}.  Moreover, $G_{u,\infty}(s)$ is strictly increasing on $(0, t_\infty^-)$, strictly decreasing on $(t_\infty^-, t_\infty^+)$ and strictly increasing on $(t_\infty^+, +\infty)$.
\item[$(3)$] If $N=3$ and $p=4$.  Then there exists unique $0<t_\infty^0<+\infty$ such that $t_\infty^0 u\in\mathcal{N}_\infty$ for $u\in\mathcal{Q}$, where $\mathcal{N}_\infty$ is given by \eqref{eqnew9020}.  Moreover, $G_{u,\infty}(s)$ is strictly increasing on $(0, t_\infty^-)$ and strictly increasing on $(t_\infty^-, +\infty)$.
\item[$(4)$] $m_\infty^->0$ in the cases $p\in(2, 4]\cap(2, 2^*)$ and $m_\infty^+<0$ in the cases $p\in(2, 4)\cap(2, 2^*)$ for $0<b<b_*(a)$, where $m_\infty^\pm=\inf_{\mathcal{N}_\infty^\pm}\mathcal{J}_\infty(u)$.
\item[$(5)$] If $u_0\in\mathcal{N}_\infty^-$ minimizes $\mathcal{J}_\infty(u)$ on $\mathcal{N}_\infty^-$ in the cases $p\in(2, 4]\cap(2, 2^*)$ for $0<b<b_*(a)$, then $u_0$ is also a critical point of $\mathcal{J}_\infty(u)$ in $\h$.
\end{enumerate}
\end{lemma}

Now, by Lemma~\ref{lemnew0007}, we can obtain the following.
\begin{proposition}\label{propnew0001}
Let $a>0$, $0<b<b_*(a)$ and $p\in(2, 4]\cap(2, 2^*)$.  Then there exists $u_0\in\mathcal{N}_\infty^-$ such that $u_0>0$, $\mathcal{J}_\infty(u_0)=m_\infty^-$ and $\mathcal{J}'_\infty(u_0)=0$ in $H^{-1}(\bbr^N)$.
\end{proposition}
\begin{proof}
By $(4)$ of Lemma~\ref{lemnew0007}, $m_\infty^-$ is well defined.  Let $\{u_n\}\subset\mathcal{N}_\infty^-$ be a minimizing sequence of $\mathcal{J}_\infty(u)$ for $m^-$.  Then by the Schwartz symmetrization, there exists $\{u_n^{*}\}\subset H^1_r(\bbr^N)$ such that
\begin{eqnarray}\label{eqnew0026}
\mathfrak{A}_{u_n^{*}}\leq\mathfrak{A}_{u_n},\quad \mathfrak{B}_{u_n^{*}}=\mathfrak{B}_{u_n}\quad\text{and}\quad\mathfrak{C}_{u_n^{*}}=\mathfrak{C}_{u_n}.
\end{eqnarray}
It follows from $\{u_n\}\subset\mathcal{N}_\infty^-$ that $G_{u_n^*,\infty}'(1)\leq0$, which together with $(1)$ and $(2)$ of Lemma~\ref{lemnew0007}, implies there exists $t_{n,\infty}^-\leq1$ such that $\{t_{n,\infty}^-u_n^{*}\}\subset\mathcal{N}_\infty^-$.  Thus, by $(1)$ and $(2)$ of Lemma~\ref{lemnew0007} and \eqref{eqnew0026} once more, we can see that
\begin{eqnarray*}
m_\infty^-+o_n(1)=\mathcal{J}_\infty(u_n)\geq\mathcal{J}_\infty(t_{n,\infty}^-u_n)\geq\mathcal{J}_\infty(t_{n,\infty}^-u_n^*)\geq m_\infty^-.
\end{eqnarray*}
Therefore, $\{t_{n,\infty}^-u_n^{*}\}\subset H^1_r(\bbr^N)\cap\mathcal{N}_\infty^-$ is also a minimizing sequence of $\mathcal{J}_\infty(u)$ for $m^-$.  Without loss of generality and for the sake of simplicity, we assume $\{u_n\}\subset H^1_r(\bbr^N)\cap\mathcal{N}_\infty^-$ is a minimizing sequence of $\mathcal{J}_\infty(u)$ for $m^-$.  For the sake of clarity, we divide the following proof into several steps.

{\bf Step.~1}\quad We prove that $\{u_n\}$ is bounded in $\h$.

Indeed, if $p=4$, then by calculating $\mathcal{J}_\infty(u_n)-\frac{1}{4}\mathcal{J}_\infty'(u_n)u_n$, we can easily to show that $\{u_n\}$ is bounded in $\h$.  Next, we consider the cases $2<p<\min\{2^*, 4\}$.  Since $\{u_n\}\subset\mathcal{N}_\infty^-$, by a similar argument as used for \eqref{eqnew9025}, we can see that $\mathcal{J}_\infty(u_n)\geq\frac{4-p}{4p}b\mathfrak{A}_{u_n}^2$.  Thus, $\{u_n\}$ is bounded in $D^{1,2}(\bbr^N)$.  On the other hand, since $\{u_n\}\subset\mathcal{N}_\infty^-$, by a similar argument as used for \eqref{eq0014}, we can see that $v_\infty\mathcal{S}^{\frac{2^*(p-2)}{2(2^*-2)}}\mathfrak{C}_u^{\frac{p-2}{2^*-2}}\leq\mathfrak{A}_u^{\frac{2^*(p-2)}{2(2^*-2)}}$.  It follows that $\{u_n\}$ is bounded in $\h$.

{\bf Step.~2}\quad We prove that there exists $u_0\in\mathcal{N}_\infty^-$ such that $\mathcal{J}_\infty(u_0)=m_\infty^-$.

Indeed, by Step.~1, there exists $u_0\in H^1_r(\bbr^N)$ such that $u_n=u_0+o_n(1)$ weakly in $H^1_r(\bbr^N)$.  Thanks to the Sobolev embedding theorem, $u_n=u_0+o_n(1)$ strongly in $L^q(\bbr^N)(2\leq q<2^*)$.  Note that $\{u_n\}\subset\mathcal{N}_\infty^-$, thus, we must have that $G_{u_0,\infty}'(1)\leq0$.  It follows from $(1)$ and $(2)$ of Lemma~\ref{lemnew0007} that there exists $t_{\infty}^-\leq1$ such that $\{t_{\infty}^-u_0\}\subset\mathcal{N}_\infty^-$.  Thus, by $(1)$ and $(2)$ of Lemma~\ref{lemnew0007} once more, we can see that
\begin{eqnarray*}
m_\infty^-+o_n(1)=\mathcal{J}_\infty(u_n)\geq\mathcal{J}_\infty(t_{\infty}^-u_n)\geq\mathcal{J}_\infty(t_{\infty}^-u_0)+o_n(1)\geq m_\infty^-+o_n(1),
\end{eqnarray*}
which implies $\mathcal{J}_\infty(t_{\infty}^-u_0)=m_\infty^-$.

Note that $|t_{\infty}^-u_0|\in\mathcal{N}_\infty^-$ and $\mathcal{J}_\infty(|t_{\infty}^-u_0|)=\mathcal{J}_\infty(t_{\infty}^-u_0)$, by $(5)$ of Lemma~\ref{lemnew0007} and the maximum principle, $|t_{\infty}^-u_0|$ is a positive solution of \eqref{eqnew9019}.
\end{proof}

Based upon Proposition~\ref{propnew0001}, we can obtain the following local compactness result.
\begin{lemma}\label{lemnew0009}
Let $a>0$, $0<b<b_*(a)$ and $p\in(2, 4]\cap(2, 2^*)$.  If the condition $(V)$ holds, then for every $(PS)_{m^-}$ sequence of $\mathcal{J}_V(u)$ contained in $\mathcal{N}_V^-$, there exists a subsequence which is compact in $\h$.
\end{lemma}
\begin{proof}
Let $\{u_n\}\subset\mathcal{N}_V^-$ is a $(PS)_{m^-}$ sequence of $\mathcal{J}_V(u)$.  Then $\mathcal{J}_V(u_n)=m^-+o_n(1)$ and $\mathcal{J}_V'(u_n)=o_n(1)$ strongly in $H^{-1}(\bbr^N)$.  For the sake of clarity, we divide the following proof into two steps.

{\bf Step.~1}\quad We prove that $m^-<m_\infty^-$.

Indeed, by Proposition~\ref{propnew0001}, there exists $u_0\in\mathcal{N}_\infty^-$ such that $u_0>0$, $\mathcal{J}_\infty(u_0)=m_\infty^-$ and $\mathcal{J}'_\infty(u_0)=0$ in $H^{-1}(\bbr^N)$.  By the condition $(V)$, we can see that $\mathcal{J}_V'(u_0)u_0<0$.  It follows from Remark~\ref{rmknew0001} and Lemmas~\ref{lemnew0002} and \ref{lemnew0003} that there exists $t^-<1$ such that $t^-u_0\in\mathcal{N}_V^-$.  Now, thanks to $(1)$ and $(2)$ of Lemma~\ref{lemnew0007}, we have
\begin{eqnarray*}
m^-\leq\mathcal{J}_V(t^-u_0)<\mathcal{J}_\infty(t^-u_0)<\mathcal{J}_\infty(u_0)=m_\infty^-.
\end{eqnarray*}

{\bf Step.~2}\quad We prove that there exists a subsequence of $\{u_n\}$ and $u^*\in\h$, still denoted by $\{u_n\}$, such that $u_n=u^*+o_n(1)$ strongly in $\h$.

Indeed, by a similar argument as used in Step.~1 of the proof to Proposition~\ref{propnew0001}, we can show that $\{u_n\}$ is bounded in $\h$.  Thus, without loss of generality, we assume that $u_n=u^*+o_n(1)$ weakly in $\h$ for some $u^*\in\h$.  Clearly, one of the following two cases must happen:
\begin{enumerate}
\item[$(a)$] $u^*=0$;
\item[$(b)$] $u^*\not=0$.
\end{enumerate}
We first consider the case $(a)$.  Since $\{u_n\}\subset\mathcal{N}_V^-$, by \eqref{eqnew9017}, we have that $\mathfrak{B}_{u_n}\geq C_1+o_n(1)$.  Thanks to the Lions lemma, there exist $R>0$ and $\{x_n\}\subset\bbr^N$ such that
\begin{eqnarray}\label{eqnew9026}
\int_{B_R(x_n)}|u_n|^2dx\geq C_2+o_n(1),
\end{eqnarray}
where $B_R(x_n)=\{x\in\bbr^N\mid|x-x_n|<R\}$.  Let $w_n=u_n(\cdot-x_n)$.  Then $\mathfrak{A}_{u_n}=\mathfrak{A}_{w_n}$ and $\mathfrak{C}_{u_n}=\mathfrak{C}_{w_n}$.  By \eqref{eqnew9026} and the Sobolev embedding theorem, we can see that $|x_n|\to+\infty$ as $n\to\infty$, $w_n=w_0+o_n(1)$ weakly in $\h$ and $w_0\not=0$.  Let $w_n^1=w_n-w_0$, then by the Brez\'is--Lieb lemma and the condition $(V)$, we have
\begin{eqnarray}\label{eqnew9027}
\mathcal{J}_V(u_n)=\mathcal{J}_\infty(w_n)+o_n(1)=\mathcal{J}_\infty(w_n^1)+\mathcal{J}_\infty(w_0)+\frac b2\mathfrak{A}_{w_n^1}\mathfrak{A}_{w_0}+o_n(1).
\end{eqnarray}
Moreover, since $\mathcal{J}_V'(u_n)=o_n(1)$ strongly in $H^{-1}(\bbr^N)$, by a standard argument, we can see that $\mathcal{J}_\infty'(w_n)=o_n(1)$ strongly in $H^{-1}(\bbr^N)$ due to the condition $(V)$.  It follows that $\mathcal{J}_\infty'(w_0)w_0\leq0$.  Clearly, there are also two cases:
\begin{enumerate}
\item[$(a1)$] $\mathcal{J}_\infty'(w_0)w_0=0$;
\item[$(a2)$] $\mathcal{J}_\infty'(w_0)w_0<0$.
\end{enumerate}
In the case $(a1)$, since $\mathcal{J}_\infty'(w_n)w_0=o_n(1)$ and $w_n=w_0+o_n(1)$ weakly in $\h$, we have that $\mathfrak{A}_{w_0}=\mathfrak{A}_{w_n}+o_n(1)$.  It follows from $w_n=w_0+o_n(1)$ weakly in $\h$ once more that $w_n=w_0+o_n(1)$ strongly in $D^{1,2}(\bbr^N)$.  Now, by applying the Sobolev embedding theorem and the fact that $(\mathcal{J}_\infty'(w_n)-\mathcal{J}_\infty'(w_0))(w_n-w_0)=o_n(1)$ in a standard way, we can see that $w_n=w_0+o_n(1)$ strongly in $\h$, which implies that $\mathcal{J}_\infty'(w_0)=0$ in $H^{-1}(\bbr^N)$.  Since $u_n\in\mathcal{N}_V^-$ and $w_n=u_n(\cdot-x_n)$, we can see from the condition $(V)$ that $G_{w_0,\infty}''(1)=G_{w_n,\infty}''(1)+o_n(1)=G_{u_n}''(1)+o_n(1)$.  We claim that $G_{u_n}''(1)\leq-C_1+o_n(1)$.  Indeed, let $w^*$ be the positive solution of $(\mathcal{P}_{1, 0})$.  Since $w^*$ is independent of $b$, by choosing $b_*(a)$ small enough if necessary and the condition $(V)$, we have
\begin{eqnarray*}
\mathcal{G}(w^*)&\leq&a\mathfrak{A}_{w^*}+\mathfrak{C}_{w^*,V}\\
&&-\frac{(4-p)a^{\frac{2(p-2)}{4-p}}}{2}\bigg(\frac{(p-2)}{2b}\bigg)^{\frac{p-2}{4-p}}\frac{\mathfrak{B}_{w^*}^{\frac{2}{4-p}}}
{(a\mathfrak{A}_{w^*}+\mathfrak{C}_{w^*,V})^{\frac{2(p-2)}{4-p}}}<0
\end{eqnarray*}
for $0<b<b_*(a)$.  By Lemmas~\ref{lemnew0002} and \ref{lemnew0003}, there exists $t^-_b>0$ such that $t^-_bw^*\in\mathcal{N}_V^-$.  Suppose $t^-_b\to+\infty$ as $b\to0^+$, then by the fact that $t^-_bw^*\in\mathcal{N}_V^-$, we can see that
\begin{eqnarray*}
0&=&(t^-_b)^4b\mathfrak{A}_{w^*}^2+(t^-_b)^2a\mathfrak{A}_{w^*}+(t^-_b)^2\mathfrak{C}_{w^*,V}-(t^-_b)^p\mathfrak{B}_{w^*}\\
&=&(t^-_b)^p\bigg((t^-_b)^{4-p}b\mathfrak{A}_{w^*}^2+(t^-_b)^{2-p}(a\mathfrak{A}_{w^*}+\mathfrak{C}_{w^*,V})-\mathfrak{B}_{w^*}\bigg)\\
&<&(t^-_b)^p\bigg(\frac{p-4}{2}\mathfrak{B}_{w^*}+(t^-_b)^{2-p}(a\mathfrak{A}_{w^*}+\mathfrak{C}_{w^*,V})\bigg)\\
&<&0
\end{eqnarray*}
for $b$ small enough.  In this inequality, we use the fact that $t^-_b<t^*$, where $t^*=\bigg(\frac{(p-2)\mathfrak{B}_{w^*}}{2b\mathfrak{A}_{w^*}^2}\bigg)^\frac{1}{4-p}$ is given by \eqref{eqnew9007}.  Thus, by choosing $b_*(a)$ small enough if necessary, we must have that $t^-_b\leq C_2$ for $0<b<b_*(a)$.  Now, by a standard argument, we can show that $\mathcal{J}_V(u_n)\leq C_3+o_n(1)$.  On the other hand, suppose that dist$(u_n, \mathcal{N}_V^+\cup\mathcal{N}_V^0)=o_n(1)$, where dist$(u, \mathcal{N}_V^+\cup\mathcal{N}_V^0)=\inf\{w\in \mathcal{N}_V^+\cup\mathcal{N}_V^0\mid (\mathfrak{A}_{u-w}+\mathfrak{C}_{u-w})^\frac12\}$.  Then by the Sobolev embedding theorem and the condition $(V)$, we must have that
\begin{eqnarray}
b\mathfrak{A}_{u_n}^2-\mathfrak{B}_{u_n}+a\mathfrak{A}_{u_n}+\mathfrak{C}_{u_n,V}=0\label{eqnew9012}\\
3b\mathfrak{A}_{u_n}^2-(p-1)\mathfrak{B}_{u_n}+a\mathfrak{A}_{u_n}+\mathfrak{C}_{u_n,V}=o_n(1).\notag
\end{eqnarray}
It follows that
\begin{eqnarray}\label{eqnew9101}
(4-p)b\mathfrak{A}_{u_n}^2=(p-2)(a\mathfrak{A}_{u_n}+\mathfrak{C}_{u_n,V})+o_n(1).
\end{eqnarray}
Now, by \eqref{eqnew9012} and \eqref{eqnew9101}, we can see that
\begin{eqnarray*}
o_n(1)+C_3&\geq&\mathcal{J}_V(u_n)\\
&=&\frac{p-2}{2p}(a\mathfrak{A}_{u_n}+\mathfrak{C}_{u_n,V})-\frac{4-p}{4p}b\mathfrak{A}_{u_n}^2\notag\\
&=&\frac{4-p}{4p}b\mathfrak{A}_{u_n}^2+o_n(1),
\end{eqnarray*}
which implies $\mathfrak{A}_{u_n}\leq C_4b^{-\frac{1}{2}}+o_n(1)$.  Note that by \eqref{eqnew9101}, we also have that $\mathfrak{A}_{u_n}\geq C_5b^{-1}+o_n(1)$.  Thus, it is a contradiction for $0<b<b_*(a)$ by choosing $b_*(a)$ small enough if necessary.  Hence, we must have that $G_{u_n}''(1)\leq-C_1+o_n(1)$.  Now, by the definition of $\mathcal{N}_\infty^-$, we have that $w_0\in\mathcal{N}_\infty^-$, which together with \eqref{eqnew9027}, implies that $m^-=m_\infty^-$.  It contradicts to Step.~1.  In the case $(a2)$, by Lemma~\ref{lemnew0007}, there exists $t^-_\infty<1$ such that $t^-_\infty w_0\in\mathcal{N}_\infty^-$.  By Lemma~\ref{lemnew0002} and a similar argument as used for \eqref{eqnew9027}, we can see that
\begin{eqnarray}
\mathcal{J}_V(u_n)&\geq&\mathcal{J}_V(t_\infty^-u_n)\notag\\
&=&\mathcal{J}_\infty(t_\infty^-w_n)+o_n(1)\notag\\
&=&\mathcal{J}_\infty(t_\infty^-w_n^1)+\mathcal{J}_\infty(t_\infty^-w_0)+o_n(1).\label{eqnew9028}
\end{eqnarray}
Clearly, there also two cases:
\begin{enumerate}
\item[$(a2_1)$] $\mathcal{J}_\infty(t_\infty^-w_n^1)\geq0$;
\item[$(a2_2)$] $\mathcal{J}_\infty(t_\infty^-w_n^1)<0$.
\end{enumerate}
In the case $(a2_1)$, we obtain that $m^-\geq m_\infty^-$ by \eqref{eqnew9028}, which is also a contradiction to Step.~1.  In the case $(a2_2)$, by Lemma~\ref{lemnew0007} once more, there exist $t_{\infty,n}^-<t_\infty^-<1$ such that $t_{\infty,n}^-w_n^1\in\mathcal{N}_\infty^-$.  Since $t_{\infty,n}^-<t_\infty^-$, we must have from Lemma~\ref{lemnew0007} once more that $\mathcal{J}_\infty(t_{\infty,n}^-w_0)\geq0$.  Now, by a similar argument as used for \eqref{eqnew9027}, we also have that
\begin{eqnarray}
\mathcal{J}_V(u_n)&\geq&\mathcal{J}_V(t_{\infty,n}^-u_n)\notag\\
&=&\mathcal{J}_\infty(t_{\infty,n}^-w_n)+o_n(1)\notag\\
&=&\mathcal{J}_\infty(t_{\infty,n}^-w_n^1)+\mathcal{J}_\infty(t_{\infty,n}^-w_0)+o_n(1),\notag
\end{eqnarray}
which implies $m^-\geq m_\infty^-$.  It also contradicts to Step.~1.  Hence, we must have the case $(b)$.  Let $w_n=u_n-u^*$.  Then $w_n=o_n(1)$ weakly in $\h$.  Moreover, by the Brez\'is--Lieb lemma and the condition $(V)$, we have
\begin{eqnarray*}
\mathcal{J}_V(u_n)=\mathcal{J}_V(u^*)+\mathcal{J}_\infty(w_n)+\frac12 \mathfrak{A}_{u^*}\mathfrak{A}_{w_n}+o_n(1).
\end{eqnarray*}
Since $\mathcal{J}_V'(u_n)=o_n(1)$ strongly in $H^{-1}(\bbr^N)$, we also have that $\mathcal{J}_V'(u^*)u^*\leq0$.  Now, thanks to Lemmas~\ref{lemnew0002} and \ref{lemnew0003}, we can apply similar arguments as used for the case $(a2)$ to obtain that $m^-\geq m_\infty^-$ if $\mathcal{J}_V'(u^*)u^*<0$.  Thus, we must have $\mathcal{J}_V'(u^*)u^*=0$.  By a similar argument as used for the case $(a1)$, we can see that $u_n=u^*+o_n(1)$ strongly in $\h$, which completes the proof.
\end{proof}

\subsection{Proof of Theorem~\ref{thm1001}}
By the Ekeland variational principle, there exists $\{u_n\}\subset\mathcal{N}_V^-$ such that
\begin{enumerate}
\item[$(1)$] $\mathcal{J}_V(u_n)=m^-+o_n(1)$;
\item[$(2)$] $\mathcal{J}_V(v)-\mathcal{J}_V(u_n)\geq-\frac1n\bigg(\mathfrak{A}_{u_n-v}+\mathfrak{C}_{u_n-v}\bigg)^{\frac12}$ for all $v\in\mathcal{N}_V^-$.
\end{enumerate}
\begin{lemma}\label{lemnew0010}
Let $a>0$, $0<b<b_*(a)$ and $p\in(2, 4]\cap(2, 2^*)$.  If the condition $(V)$ holds, then there exist $\ve_n>0$ and $t_n(l):[-\ve_n, \ve_n]\to[\frac12, \frac32]$ such that $t_n(l)u_n+lw\in\mathcal{N}_V^-$ for all $w\in\mathbb{B}_1:=\{u\in\h\mid\mathfrak{A}_{u}+\mathfrak{C}_{u}=1\}$.  Moreover, $t_n(l)$ are of $C^1$ and
\begin{eqnarray}\label{eqnew9030}
t_n'(0)=\frac{(4b\mathfrak{A}_{u_n}+2a)\langle\nabla u_n, \nabla w\rangle+2\langle u_n, w\rangle-p\int_{\bbr^N}|u_n|^{p-2}u_nwdx}{3b\mathfrak{A}_{u_n}+a\mathfrak{A}_{u_n}+\mathfrak{C}_{u_n,V}-\mathfrak{B}_{u_n}},
\end{eqnarray}
where $\langle\cdot, \cdot\rangle$ is the usual inner product in $L^2(\bbr^N)$.
\end{lemma}
\begin{proof}
Let
\begin{eqnarray*}
\mathcal{T}_n(t,l)=b\mathfrak{A}_{tu_n+lw}^2+a\mathfrak{A}_{tu_n+lw}+\mathfrak{C}_{tu_n+lw,V}-\mathfrak{B}_{tu_n+lw},
\end{eqnarray*}
where $w\in\mathbb{B}_1$.  Clearly, $\mathcal{T}_n(1,0)=0$.  By applying the implicit function theorem to $\mathcal{T}_n(t,l)$, we can show that there exist $\ve_n>0$ and $t_n(l):[-\ve_n, \ve_n]\to[\frac12, \frac32]$ such that $t_n(l)u_n+lw\in\mathcal{N}_V$ for all $w\in\mathbb{B}_1$.  Moreover, $t_n(l)$ are of $C^1$ and $t_n'(0)$ satisfy \eqref{eqnew9030}.  It remains to show that $t_n(l)u_n+lw\in\mathcal{N}_V^-$.  Indeed, by a similar argument as used in the proof of Lemma~\ref{lemnew0009}, we can see that $3b\mathfrak{A}_{u_n}^2-(p-1)\mathfrak{B}_{u_n}+a\mathfrak{A}_{u_n}+\mathfrak{C}_{u_n,V}\leq-C_1+o_n(1)<0$.  Now, by choosing $\ve_n$ small enough if necessary, we actually have that $t_n(l)u_n+lw\in\mathcal{N}_V^-$.
\end{proof}

Now, we can give the proof of Theorem~\ref{thm1001}.

\noindent\textbf{Proof of Theorem~\ref{thm1001}.}\quad Let $v$ in $(2)$ be $t_n(l)u_n+lw$, then we have
\begin{enumerate}
\item[$(1)$] $\mathcal{J}_V(u_n)=m^-+o_n(1)$;
\item[$(2)$] $\mathcal{J}_V(t_n(l)u_n+lw)-\mathcal{J}_V(u_n)\geq-\frac1n\bigg(\mathfrak{A}_{(t_n(l)-1)u_n+lw}+\mathfrak{C}_{(t_n(l)-1)u_n+lw}\bigg)^{\frac12}$.
\end{enumerate}
By a similar argument as used in the proof of Step.~1 to Proposition~\ref{propnew0001}, we can show that $\{u_n\}$ is bounded in $\h$.  Thanks to \eqref{eqnew9030}, we actually have that $|t_n(0)|\leq C_2(\mathfrak{A}_w+\mathfrak{C}_w)^{\frac12}$.  It follows from $(2)$, Lemma~\ref{lemnew0010} and a standard argument that $\mathcal{J}_V'(u_n)=o_n(1)$ strongly in $H^{-1}(\bbr^N)$.  By Lemma~\ref{lemnew0009}, $u_n=u_*+o_n(1)$ strongly in $\h$ for some $u_*$ up to a subsequence.  It follows from a similar argument as used in the proof of Lemma~\ref{lemnew0009} that $u_*\in\mathcal{N}_V^-$ and $\mathcal{J}_V(u_*)=m^-$.  Note that $|u_*|\in\mathcal{N}_V^-$ and $\mathcal{J}_V(|u_*|)=m^-$, thus, by Lemmas~\ref{lemnew0005} and \ref{lemnew0006}, $|u_*|$ is a nonnegative solution to $(\mathcal{P}_{a, b})$.  Thanks to the maximum principle, $(\mathcal{P}_{a, b})$ has a positive solution for $a>0$, $p\in(2, 2^*)\cap(2, 4]$ and $0<b<b_*(a)$.  It remains to show that $(\mathcal{P}_{a, b})$ only has trivial solution for $a>0$, $p\in(2, 2^*)\cap(2, 4)$ and $b$ large enough in the case $N\geq4$.  Indeed, let $u$ be a nontrivial solution of $(\mathcal{P}_{a, b})$ in the cases $N\geq4$, then by a similar argument as used for \eqref{eqnew9017}, we can see that
\begin{eqnarray}
0&=&b\mathfrak{A}_{u}^2+a\mathfrak{A}_{u}+\mathfrak{C}_{u,V}-\mathfrak{B}_{u}\notag\\
&\geq&b\mathfrak{A}_{u}^2-C_1\mathfrak{A}_{u}^{\frac{2^*}{2}}.\label{eqnew9031}
\end{eqnarray}
Since $2^*=4$ for $N=4$, \eqref{eqnew9031} is impossible for $b$ large enough.  In the cases $N\geq5$, we have that $2^*<4$.  Thus, by \eqref{eqnew9031} and the Young inequality, we can see that
\begin{eqnarray}\label{eqnew9032}
0\geq\frac b2\mathfrak{A}_{u}^2-C_2 b^{-\frac{2^*}{4-2^*}}.
\end{eqnarray}
Note that by a similar argument as used for \eqref{eqnew9017}, we also have that $\mathfrak{A}_{u}\geq C_3$.  Thus, \eqref{eqnew9032} is impossible for $b$ large enough.
\qquad\raisebox{-0.5mm}{%
\rule{1.5mm}{4mm}}\vspace{6pt}

\section{Acknowledgement}
Y. Huang was Natural Science Foundation of China (11471235 and 11171247).  Z. Liu was
supported by Suzhou University of Science and Technology foundation grant (331412104).  Y. Wu was supported by the Fundamental Research Funds for the Central
Universities (2014QNA67).


\begin{thebibliography}{999}
\bibitem{A12}
A. Azzollini, The Kirchhoff equation in $\bbr^3$ perturbed by a local
nonlinearity,  {\it Differential Integral Equations,} {\bf 25} (2012), 543--554.

\bibitem{A13}
A. Azzollini, A note on the elliptic Kirchhoff equation in $\bbr^N$ perturbed by a local nonlinearity, arXiv:1306.2064v1[math.AP].

\bibitem{ACM05}
C. O. Alves, F. J. A. Corr\'ea, To Fu Ma, Positive solutions for a quasilinear elliptic equation of Kirchhoff type, {\it Comput. Math. Appl.,} {\bf 49} (2005), 85--93.

\bibitem{AF12}
C. O. Alves, G. M. Figueiredo, Nonlinear perturbations of a periodic Kirchhoff equation in $\bbr^3$, {\it Nonlinear Anal. TMA,} {\bf 75} (2012), 2750--2759.

\bibitem{CR92}
M. Chipot, J. F. Rodrigues, On a class of nonlocal nonlinear problems, {\it RAIRO Modelisation Math. Anal.
Numer.} {\bf26} (1992), 447--467.

\bibitem{CL97}
M. Chipot, B. Lovat, Some remarks on nonlocal elliptic
and parabolic problems, {\it Nonlinear Anal.} {\bf30} (1997), 4619--4627.

\bibitem{CWL12}
B. Cheng, X. Wu, J. Liu, Multiple solutions for a class of Kirchhoff type problems with concave nonlinearity, {\it NoDEA Nonlinear Differ. Equ. Appl.,} {\bf 19} (2012), 521--537.

\bibitem{CKW11}
C. Chen, Y. Kuo, T.-F. Wu, The Nehari manifold for a Kirchhoff type problem involving
sign-changing weight functions, {\it J. Differential Equations,} {\bf 250} (2011), 1876--1908.

%\bibitem{FJ12}
%G. M. Figueiredo,  Jo\~ao R. Santos Junior, Multiplicity of solutions for a Kirchhoff equation with subcritical or critical growth, {\it Differential Integral Equations,} {\bf25} (2012),  853--868.

%\bibitem{F13}
%G. M. Figueiredo, Existence of a positive solution for a Kirchhoff problem type with critical growth via truncation
%argument, {\it J. Math. Anal. Appl.,} {\bf401} (2013), 706--713.

\bibitem{G15}
Z. Guo, Ground states for Kirchhoff equations without compact condition, {\it J. Differential Equations,} {\bf259} (2015) 2884--2902.

\bibitem{HZ12}
X. He, W. Zou, Existence and concentration behavior of positive solutions
for a Kirchhoff equation in $\bbr^3$, {\it J. Differential Equations,} {\bf 252} (2012), 1813--1834.

\bibitem{HLP14}
Y. He, G. Li, S. Peng, Concentrating bound states for Kirchhoff type problems in $\bbr^3$ involving critical Sobolev exponents, {\it Adv. Nonlinear Stud.,} {\bf14} (2014), 441--468.

%\bibitem{HLW13}
%Y. Huang, Z. Liu, Y. Wu, Existence of prescribed $L^2$-Norm solutions for a class of Schr\"odinger-Poisson equation, {\it Abst. Appl. Anal.,} {\bf 2013} (2013) 398164 (11 pages).

\bibitem{HLW15}
Y. Huang, Z. Liu, Y. Wu, On finding solutions of a Kirchhoff type problem, {\it Proc. Amer. Math. Soc.,} DOI: 10.1090/proc/12946.

\bibitem{HLW151}
Y. Huang, Z. Liu, Y. Wu, On Kirchhoff type equations with critical Sobolev exponent and Naimen's open problems, arXiv:1507.05308v1  [math.AP].

\bibitem{HWW15}
Y. Huang, T.-F. Wu, Y. Wu, Multiple positive solutions for a class of
concave-convex elliptic problems in $\mathbb{R}^N$ involving sign-changing
weight (II), {\it Comm. Contemp. Math.,}  {\bf17} (2015) 1450045 (35 pages).

\bibitem{J99}
L. Jeanjean, On the existence of bounded Palais-Smale sequences and application to a Landsman-Lazer-type problem set on $\bbr^N$, {\it Proc. Roy. Soc. Edinburgh Sect. A,} {\bf129}(1999), 787--809.

\bibitem{K83}
G. Kirchhoff, Mechanik. Teubner, Leipzig (1883).

%\bibitem{K89}
%M. Kwong, Uniqueness of positive solution of $\Delta u-u+u^p= 0$ in $\bbr^N$, {\it Arch.
%Rat. Math. Anal.,} {\bf105} (1989), 243--266.

\bibitem{L84}
P. Lions, The concentration-compactness principle in the calculus of variations. The locally compact case, {\it I, Ann.
Inst. H. Poincar\'e Anal. Non Lin\'eaire,} {\bf1} (1984), 109--145.

\bibitem{LLS12}
Y. Li, F. Li, J. Shi, Existence of a positive solution to Kirchhoff type problems
without compactness conditions, {\it J. Differential Equations,} {\bf 253} (2012), 2285--2294.

\bibitem{LLS14}
Z. Liang, F. Li, J. Shi, Positive solutions to Kirchhoff type equations with nonlinearity
having prescribed asymptotic behavior, {\it Ann. Inst. H. Poincar\'e Anal. Non Lin\'eaire,} {\bf31}(2014), 155--167.

\bibitem{LY14}
G. Li, H. Ye, Existence of positive ground state solutions for the nonlinear Kirchhoff type equations in $\bbr^3$, {\it J. Differential Equations,} {\bf 257} (2014), 566--600.

%\bibitem{LLT151}
%C. Lei, J. Liao, C. Tang, Multiple  positive solutions for Kirchhoff type of problems with singularity and critical exponents, {\it J. Math. Anal. Appl.,} {\bf421} (2015), 521--538.

\bibitem{LLT152}
J. Liu, J. Liao, C. Tang, Positive solutions for Kirchhoff-type equations with critical exponent in $\bbr^N$, {\it J. Math. Anal. Appl.,} {\bf429} (2015), 1153--1172.

%\bibitem{N14}
%D. Naimen, Positive solutions of Kirchhoff type elliptic equations involving a critical Sobolev exponent, {\it NoDEA Nonlinear Differ. Equ. Appl.,} {\bf 21} (2014), 885--914.

\bibitem{N141}
D. Naimen, The critical problem of Kirchhoff type elliptic equations
in dimension four, {\it  J. Differential Equations, } {\bf257} (2014), 1168--1193.

\bibitem{PZ06}
K. Perera, Z. Zhang, Nontrivial solutions of Kirchhoff-type problems via the Yang index, {\it J. Differential Equations,} {\bf221} (2006) 246--255.

\bibitem{R06}
D. Ruiz, The Schr\"odinger-Poisson equation under the effect of a nonlinear local term, {\it J. Funct.
Anal.,} {\bf237} (2006) 655--674.
%
%\bibitem{RS10}
%D. Ruiz, G. Siciliano, Existence of ground states for a modified nonlinear Schr\"odinger
%equation, {\it Nonlinearity,} {\bf23} (2010) 1221--1233.


\bibitem{R15}
B. Ricceri, Energy functionals of Kirchhoff-type problems having multiple global minima, {\it Nonlinear Anal. TMA,} {\bf115} (2015) 130--136.

%\bibitem{S85}
%J. Shatah, Unstable ground state of nonlinear Klein-Gordon equations, {\it Trans. Amer. Math. Soc.,} {\bf290} (1985), 701--710.

\bibitem{S85}
M. Struwe, On the evolution of harmonic mappings of Riemannian surfaces, {\it Comment. Math. Helv.,} {\bf60}(1985)
558--581.

\bibitem{SW14}
J. Sun, T.-F. Wu, Ground state solutions for an indefinite Kirchhoff type problem with steep potential well,   {\it  J. Differential Equations, }  {\bf 256}(2014),  1771--1792.

%\bibitem{WTXZ12}
%J. Wang, L. Tian, J. Xu, F. Zhang, Multiplicity and concentration of positive solutions
%for a Kirchhoff type problem with critical growth, {\it J. Differential Equations,} {\bf 253} (2012), 2314--2351.

\bibitem{WHL15}
Y. Wu, Y. Huang, Z. Liu, On a Kirchhoff type problem in $\bbr^N$, {\it J. Math. Anal. Appl.,} {\bf425} (2015), 548--564.

\bibitem{WHL151}
Y. Wu, Y. Huang, Z. Liu, On a Kirchhoff type problems with potential well and indefinite potential, arXiv:1507.03373v1  [math.AP].

\bibitem{W15}
W. Shuai, Sign-changing solutions for a class of Kirchhoff-type
problem in bounded domains, {\it J. Differential Equations,} {\bf259} (2015), 1256--1274.

%\bibitem{ZP06}
%Z. Zhang, K. Perera, Sign changing solutions of Kirchhoff type problems via invariant sets of descent flow, {\it J. Math. Anal. Appl.,} {\bf 317} (2006), 456--463.
\end{thebibliography}
\end{document}